\documentclass[a4, 12pt, leqno]{article}
\usepackage{amsmath}
\usepackage{amssymb}
\usepackage{amsthm}
\usepackage{overpic}
\newtheorem{df}{Definition}[section]
\newtheorem{pr}[df]{Proposition}
\newtheorem{Th}[df]{Theorem}

\newtheorem{lm}[df]{Lemma}

\newtheorem{as}[df]{Assumption}

\begin{document}

	\title{\bf Long-time behavior of an Arc-shaped Vortex Filament and
	its Application to the Stability of a Circular Vortex Filament}
\author{Masashi A{\sc iki}}
\date{}
\maketitle
\vspace*{-0.5cm}

\begin{abstract}
We consider a nonlinear model equation, known as the 
Localized Induction Equation, describing the motion of a vortex filament
immersed in an incompressible and inviscid fluid. 
We show stability estimates for an arc-shaped vortex filament, 
which is an exact solution to an initial-boundary value problem for the 
Localized Induction Equation.
An arc-shaped filament travels along an axis at a constant speed
without changing its shape, and is
oriented in such a way that the arc stays in a plane that is
perpendicular to the axis. 
We prove that an arc-shaped filament is stable in the 
Lyapunov sense for general perturbations except in the axis-direction,
for which the perturbation can grow linearly in time.
We also show that this estimate is optimal.
We then apply the obtained stability estimates to study the 
stability of a circular vortex filament
under some symmetry assumptions on the initial perturbation.
We do this by dividing the circular filament into arcs, apply the stability
estimate to each arc-shaped filament, and combine the estimates to 
obtain estimates for the whole circle.
The optimality of the stability estimates
for an arc-shaped filament also shows that
a circular filament is not stable in the Lyapunov sense, namely, 
certain perturbations can grow linearly in time.
 
\end{abstract}

\section{Introduction and Problem Setting}

A vortex filament is a space curve on which the vorticity of the fluid is concentrated. 
Vortex filaments are used to model very thin vortex structures such as vortices that 
trail off airplane wings or propellers. 
Another common vortex structure modeled by a vortex filament is a 
vortex ring, which is a thin torus-shaped vortex structure, such as a
bubblering. A vortex ring would be described as a circular vortex filament.
In this paper, we consider a vortex filament immersed in an incompressible 
and inviscid fluid,
for which its motion can be modeled by the  
Localized Induction Equation given by 
\begin{align*}
\mbox{\mathversion{bold}$x$}_{t} =\mbox{\mathversion{bold}$x$}_{s}\times 
\mbox{\mathversion{bold}$x$}_{ss},
\end{align*}
where \( \mbox{\mathversion{bold}$x$}(s, t)= {}^{t}(x_{1}(s ,t), x_{2}(s ,t), x_{3}(s ,t))\) is the 
position vector of the vortex filament parametrized by 
its arc length \( s \) at time \( t\), 
\( \times \) is the exterior product in the three-dimensional Euclidean space,
and subscripts \( s\) and \(t \) are differentiations with the respective variables.
The 
Localized Induction Equation (LIE) is derived by applying the so-called localized induction approximation to the 
Biot--Savart integral, which gives the velocity field of the fluid from the vorticity field.
The LIE was first derived by 
Da Rios \cite{20} in 1906 and was re-derived twice independently by Murakami et al. \cite{22} in 1937 and by 
Arms and Hama \cite{21} in 1965.

In this paper, we consider the stability of an arc-shaped filament 
and a circular vortex filament, which are exact solutions of the LIE.
Before we go into the details of the specific filaments, 
we go over some backgrounds on the relevant problems.
An arc-shaped filament is an exact solution of the following
initial-boundary value problem for the LIE.
\begin{eqnarray}
\left\{
\begin{array}{ll}
\mbox{\mathversion{bold}$x$}_{t} =\mbox{\mathversion{bold}$x$}_{s}\times 
\mbox{\mathversion{bold}$x$}_{ss}, & s\in I_{L}, \ t>0, \\[3mm]
\mbox{\mathversion{bold}$x$}(s,0)=\mbox{\mathversion{bold}$x$}_{0}(s), & s\in I_{L}, \ t>0, \\[3mm]
\mbox{\mathversion{bold}$x$}_{s}(0,t)=\mbox{\mathversion{bold}$a$}, \ 
\mbox{\mathversion{bold}$x$}_{s}(L,t)=\mbox{\mathversion{bold}$e$}_{3}, & t>0,
\end{array}\right.
\label{slant}
\end{eqnarray}
where 
\( I_{L}=(0,L)\subset \mathbf{R}\) is an open interval, 
\( \mbox{\mathversion{bold}$a$}\in \mathbf{R}^{3}\) is an arbitrary vector satisfying 
\( |\mbox{\mathversion{bold}$a$}|=1 \), and \( \mbox{\mathversion{bold}$e$}_{3}={}^{t}(0,0,1) \).
Note that since the independent variable \( s\) is the arc length, 
\( L\) is determined by the length of the initial filament.
We can see that, by taking the trace \( s=0\) in the equation of problem
(\ref{slant}), a filament moving according to problem (\ref{slant}) satisfies
\begin{align*}
\mbox{\mathversion{bold}$x$}_{t}(0,t)&=\mbox{\mathversion{bold}$a$}\times \mbox{\mathversion{bold}$x$}_{ss}(0,t),
\end{align*}
which means that the end-point \( \mbox{\mathversion{bold}$x$}(0,t)\)
of the filament moves along the plane perpendicular to \( \mbox{\mathversion{bold}$a$}\).
Hence, problem (\ref{slant}) describes a vortex filament moving on a slanted plane, namely, 
the plane perpendicular to \( \mbox{\mathversion{bold}$a$} \). 
The reason we also imposed a boundary condition at \( s=L \) is for the following reason. 
A more intuitive problem setting for a vortex filament moving on a plane would be
\begin{eqnarray}
\left\{
\begin{array}{ll}
\mbox{\mathversion{bold}$x$}_{t} =\mbox{\mathversion{bold}$x$}_{s}\times 
\mbox{\mathversion{bold}$x$}_{ss}, & s>0, \ t>0, \\[3mm]
\mbox{\mathversion{bold}$x$}(s,0)=\mbox{\mathversion{bold}$x$}_{0}(s), & s>0, \ t>0, \\[3mm]
\mbox{\mathversion{bold}$x$}_{s}(0,t)=\mbox{\mathversion{bold}$a$}, & t>0,
\end{array}\right.
\label{hp}
\end{eqnarray}
which is a problem describing an infinitely long filament with one end moving along the plane perpendicular to 
\( \mbox{\mathversion{bold}$a$}\). The solvability of problem (\ref{hp}) is a direct consequence of 
a previous work by the author and Iguchi \cite{11}, which proved the solvability of problem 
(\ref{hp}) with \( \mbox{\mathversion{bold}$a$}=\mbox{\mathversion{bold}$e$}_{3} \), because the solution of
problem (\ref{hp}) can be obtained by rotating the solution obtained in \cite{11} in a way that 
\( \mbox{\mathversion{bold}$a$} \) is trasformed to \( \mbox{\mathversion{bold}$e$}_{3} \). Hence,
problem (\ref{hp}) for general \( \mbox{\mathversion{bold}$a$} \) is essentially the same as the case 
\( \mbox{\mathversion{bold}$a$}=\mbox{\mathversion{bold}$e$}_{3} \). 
So to describe the motion of a vortex filament on a slanted plane, we imposed a boundary condition at 
\( s=L \) to set a reference plane which allows us to express the slanted-ness of the plane that the 
filament is moving on.

The time-global solvability of problem (\ref{slant}) was proved by the author in \cite{23}.
Since the problem is invariant under rotation and translation, 
if \( \mbox{\mathversion{bold}$a$} \) and 
\( \mbox{\mathversion{bold}$e$}_{3}\) are linearly independent, we can always transform the problem to the following.
\begin{align}
\left\{
\begin{array}{ll}
\mbox{\mathversion{bold}$x$}_{t} =\mbox{\mathversion{bold}$x$}_{s}\times 
\mbox{\mathversion{bold}$x$}_{ss}, & s\in I_{L}, \ t>0, \\[3mm]
\mbox{\mathversion{bold}$x$}(s,0)=\mbox{\mathversion{bold}$x$}_{0}(s), & s\in I_{L}, \ t>0, \\[3mm]
\mbox{\mathversion{bold}$x$}_{s}(0,t)=\mbox{\mathversion{bold}$b$}, \ 
\mbox{\mathversion{bold}$x$}_{s}(L,t)=\mbox{\mathversion{bold}$e$}_{2}, & t>0,
\end{array}\right.
\label{slant2}
\end{align}
where \( \mbox{\mathversion{bold}$e$}_{2}={}^{t}(0,1,0)\), \( \mbox{\mathversion{bold}$b$}=
{}^{t}(b_{1},b_{2},0)\) satisfies \( |\mbox{\mathversion{bold}$b$}|=1 \), and is oriented so that 
the plane perpendicular to \( \mbox{\mathversion{bold}$b$}\) forms an angle \( \theta \)
 in the 
counter-clockwise direction, from the plane
perpendicular to \( \mbox{\mathversion{bold}$e$}_{2}\) for some 
\( \theta \in (0,2\pi )\). Formulated in this way, we can regard 
\( \theta \in (0,2\pi )\) as the given data instead of \( \mbox{\mathversion{bold}$b$}\)
since \( \theta \) will uniquely determine 
\( \mbox{\mathversion{bold}$b$}\) by 
\( \mbox{\mathversion{bold}$b$}={}^{t}(-\sin \theta , \cos \theta , 0 )\).
For convenience, we denote the plane at \( s=0\) as the lower plane, and the plane at \( s=L\) as the upper plane. We also set the origin of the Euclidean space at the intersection 
of the lower and upper plane such that the line created by this intersection coincides with 
the axis \( \{ {}^{t}(0,0,\xi_{3}) \ | \ \xi_{3} \in \mathbf{R} \}\). From here, \( \xi_{i} \ (i=1,2,3)\) will be used to  
denote the axis of the three-dimensional Euclidean space. We make this notation to avoid confusion with the 
components of the position vector of the filament.
In what follows, we consider
problem (\ref{slant2}) instead of problem (\ref{slant}) without loss of generality. 

Many mathematical research has been done on the LIE.
Nishiyama and Tani \cite{5,13} proved the unique solvability of the initial value problem in Sobolev 
spaces. Koiso \cite{12} considered a geometrically generalized setting in which he rigorously proved
the equivalence of the solvability of initial value problems for the LIE and the cubic nonlinear Schr\"odinger equation. 
This equivalence was first shown by 
Hasimoto \cite{14} in which he studied the formation of solitons on a vortex filament. He defined a transformation of variable known as the Hasimoto transformation to transform the LIE into the cubic
nonlinear Schr\"odinger equation. 
The Hasimoto transformation is a change of variable given by 
\[ \psi(s,t) = \kappa(s,t) \exp \left( {\rm i} \int ^{s}_{0} \tau(r,t) \,{\rm d}r \right), \]
where \( {\rm i}\) is the imaginary unit, \( \kappa \) is the curvature,
 and \( \tau \) is the torsion of the filament.
Defined as such, it is well known that \( \psi \) satisfies the 
nonlinear Schr\"odinger equation given by 
\begin{eqnarray}
{\rm i}\frac{\partial \psi}{\partial t}
= \frac{ \partial ^{2}\psi }{\partial s^{2}} + \frac{1}{2} \left| \psi \right| ^{2}\psi .
\label{NLS}
\end{eqnarray}
The original transformation proposed by Hasimoto uses the torsion of the filament in its definition,
which means that the transformation is undefined at points where the curvature of the filament is zero. 
Koiso \cite{12} constructed a transformation, sometimes referred to as the 
generalized Hasimoto transformation, and gave a mathematically rigorous 
proof of the equivalence of the solvability of initial value problem for the LIE and (\ref{NLS}). 

More recently, 
Banica and Vega \cite{15,16,25} and Guti\'errez, Rivas, and Vega \cite{17} constructed and analyzed
a family of self-similar solutions of the LIE which forms a corner in finite time. 
The authors \cite{11} proved the unique solvability of an initial-boundary value problem for the LIE
in which the filament moved in the three-dimensional half space. Nishiyama and Tani \cite{5} also
considered initial-boundary value problems with different boundary conditions. 

As far as the author knows, the series of papers \cite{15,16,25,90,17} are the only ones that 
obtain stability estimates for the position of the filament moving under the LIE.
Thanks to the result by Koiso \cite{12}, it is possible to transform the 
initial value problem for the LIE into the initial value problem for the
nonlinear Schr\"odinger equation, and approach the stability of a vortex filament from
the analysis of the nonlinear Schr\"odinger equation. As is well known, there are many
 results about the stability of 
specific solutions (for example, plane wave solutions and ground state solutions)
of the nonlinear Schr\"odinger equation, and although these are related to 
the stability of the motion of a vortex filament, we mention that obtaining stability estimates for the 
solution of the Schr\"odinger equation does not imply estimates for the position vector of the 
original filament. This is because, roughly speaking, estimates for the solution of the Schr\"odinger equation 
correspond to estimates of the curvature of the filament, and it is non-trivial
(if at all possible) to obtain estimates for the position of the filament from the estimates for the curvature. 
Indeed, in \cite{15,16,25,17}, they were able to obtain estimates for the position of the 
filament by utilizing the self-similar nature of the filament alongside estimates obtained from the 
analysis of the Schr\"odinger equation through the Hasimoto transformation.

The stability analysis in this paper is much more simple and direct
compared to previous results. We don't use the Hasimoto transformation,
and the estimates for the perturbation is obtained by utilizing 
conserved quantities and direct energy estimates for the perturbation,
as well as symmetries that are conserved throughout the motion of the 
filament.
The relation between problem (\ref{slant2}) and the corresponding problem for the
nonlinear Schr\"odinger equation is addressed by the author in more detail in \cite{40}.

\medskip

\medskip

We return to the problem at hand.
For \( R>0\), a particular exact solution of problem (\ref{slant2}) is given by 
\begin{align*}
\mbox{\mathversion{bold}$x$}^{R}(s,t) := {}^{t}\big( R\cos (\frac{s}{R}), R\sin (\frac{s}{R}), 
\frac{t}{R} \big),
\end{align*}
which is an arc-shaped filament with radius \( R\)
 spanned between the lower and upper plane travelling at a 
constant speed \( \frac{1}{R}\). The center of the arc is located at the 
intersection of the lower and upper planes.
We also set
\begin{align*}
\mbox{\mathversion{bold}$x$}^{R}_{0}(s) := {}^{t}\big( R\cos (\frac{s}{R}), R\sin (\frac{s}{R}), 
0 \big).
\end{align*}

We want to consider the stability of the above arc-shaped filament, and hence 
we consider the solution of problem (\ref{slant2}) with initial datum given by 
\(
 \mbox{\mathversion{bold}$x$}^{R}_{0}+ 
\mbox{\mathversion{bold}$\varphi $}_{0}
\),
where \( \mbox{\mathversion{bold}$\varphi $}_{0} \) is the initial perturbation, and 
look at the behavior of \( \mbox{\mathversion{bold}$\varphi $}(s,t) :=
\mbox{\mathversion{bold}$x$}(s,t)-\mbox{\mathversion{bold}$x$}^{R}(s,t) \).
We must first observe what kind of stability is suitable for our problem.
A simple counter-example shows that, regardless of how small we take the initial perturbation,
asymptotic stability doesn't hold, i.e. 
\( \mbox{\mathversion{bold}$\varphi $}\) doesn't vanish in general as \( t\to \infty \).
Namely, for any \( \varepsilon >0\), if we consider an initial perturbation of the form 
\( {}^{t}(0,0, \varepsilon )\), then the corresponding solution
\( \mbox{\mathversion{bold}$x$}(s,t)\) of (\ref{slant2}) is
explicitly given by
\begin{align*}
\mbox{\mathversion{bold}$x$}(s,t)=\mbox{\mathversion{bold}$x$}^{R}(s,t)
+
\mbox{\mathversion{bold}$\varphi$}_{0}(s),
\end{align*}
which is separated by a constant distance \( \varepsilon \) from 
\( \mbox{\mathversion{bold}$x$}^{R}(s,t) \) for all \( s\in I_{L} \) and \( t>0 \).
As we will show later, we can actually prove that for any initial perturbation,
\( \mbox{\mathversion{bold}$\varphi $} \) will not vanish (unless of course, the 
initial perturbation is a zero vector) as \( t\to \infty\). 
Hence we will focus on proving Lyapunov type 
stability for the arc-shaped filament. We summarize the problem setting in the 
following figure since a lot of what was explained above is 
easier to understand as a picture rather than words.

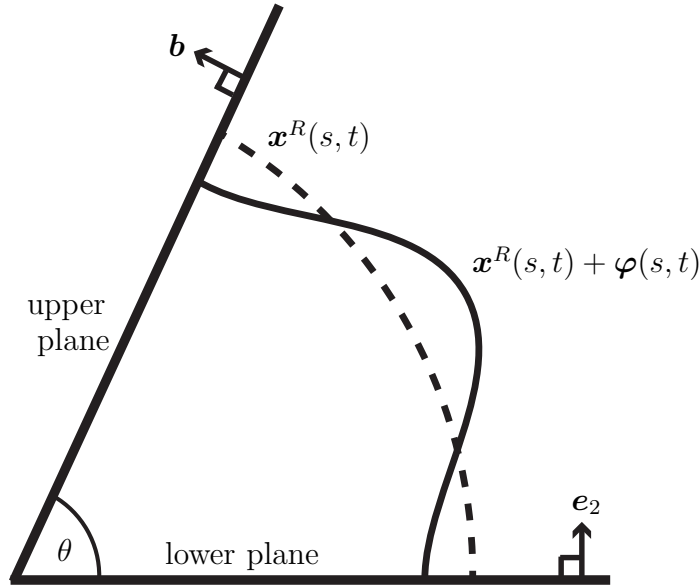
\begin{figure}[h] 
\begin{center}
\begin{overpic}[width=0.5\textwidth]{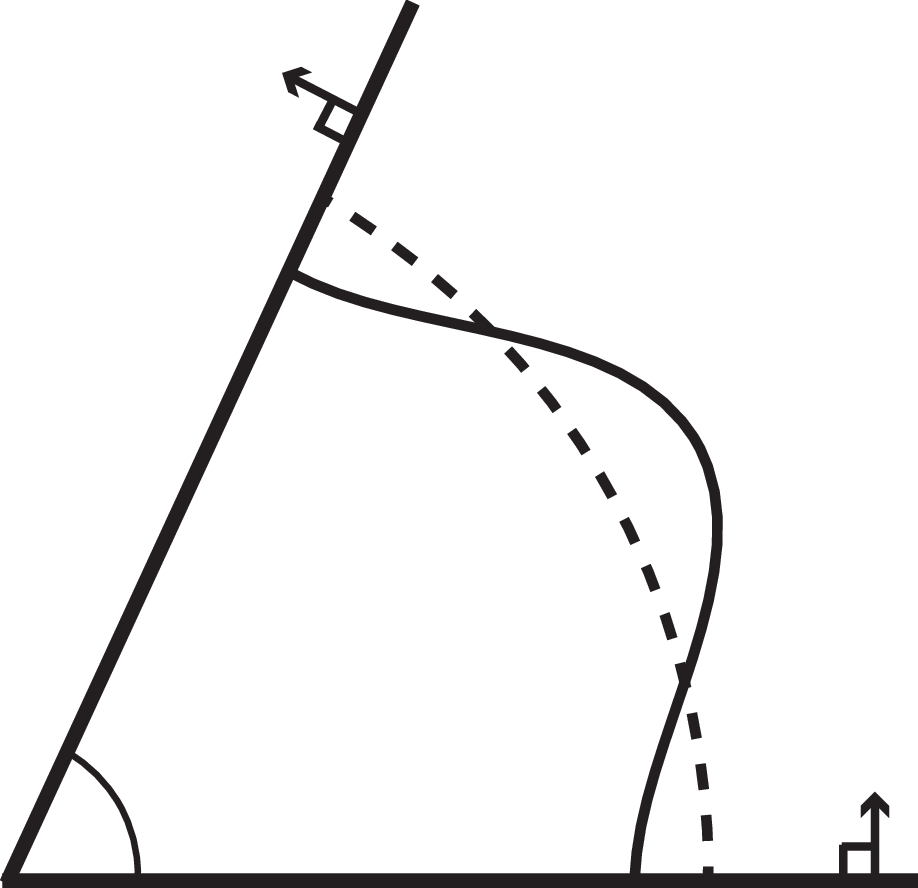}
\put(26.5,88){\( \mbox{\mathversion{bold}$b$}\) }
\put(93.5,12.5){\( \mbox{\mathversion{bold}$e$}_{2}\) }
\put(8,4){\( \theta \)}
\put(43,73){\( \mbox{\mathversion{bold}$x$}^{R}(s,t) \)}
\put(77,52){\( \mbox{\mathversion{bold}$x$}^{R}(s,t)
+ \mbox{\mathversion{bold}$\varphi $}(s,t) \)}
\put(3,45){upper}
\put(3,39){ plane}
\put(26,3.5){lower plane}
\end{overpic}
\end{center}
\caption{A schematic of the problem setting for the initial-boundary value problem (\ref{slant2})}
\end{figure}

\medskip

\medskip

We also consider the stability of a circular vortex filament,
which is an exact solution of the following initial value problem.
\begin{align}
\left\{
\begin{array}{ll}
\mbox{\mathversion{bold}$x$}_{t}=\mbox{\mathversion{bold}$x$}_{s}
\times \mbox{\mathversion{bold}$x$}_{ss}, & s\in \mathbf{T}, \ t>0, \\[3mm]
\mbox{\mathversion{bold}$x$}(s,0)=\mbox{\mathversion{bold}$x$}_{0}(s), & s\in \mathbf{T}, \ t>0,
\end{array}\right.
\label{closed}
\end{align}
where \( \mathbf{T}=\mathbf{R}/[0,L] \).
Problem \ref{closed} describes the motion of a closed vortex filament. 
A particular solution of problem \ref{closed} is a circular vortex filament
\( \mbox{\mathversion{bold}$x$}^{R}\) given by
\begin{align*}
\mbox{\mathversion{bold}$x$}^{R}(s,t)=
{}^{t}\big( R\cos (\frac{s}{R}), R\sin (\frac{s}{R}), 
\frac{t}{R} \big),
\end{align*}
where the initial filament \( \mbox{\mathversion{bold}$x$}^{R}_{0}\)
is given by 
\begin{align*}
\mbox{\mathversion{bold}$x$}^{R}_{0}(s)=
{}^{t}\big( R\cos (\frac{s}{R}), R\sin (\frac{s}{R}), 
0 \big).
\end{align*}
Note that the above notation is the same as the arc-shaped filament. 
Since the expression for both are the same, we will differentiate between the 
two through context rather than notation.

There is a long history of research 
\cite{80, 101, 61, 51, 55, 64, 57, 100, 77, 52, 50, 76, 
75, 59, 69, 54, 62, 53, 66, 73, 58, 60, 70, 67, 68, 65, 78, 72, 79} studying the motion and stability of 
vortex rings through various methods such as experiments, numerics, and theoretical 
analysis. 
Many of these researches focus on the instability,
linear stability, and orbital stability of the motion of vortex rings, under
various perturbations and deformations. 
Here, orbital stability of a vortex ring refers to the stability of the 
ring to perturbations up to a translation along the axis of the ring.
The complexity and rich nature of the motion of a vortex ring and how it 
relates to the motion of fluids in general has attracted many researchers,
and still is one of the most fundamental subject of research in fluid dynamics.
A mathematical treatment of the stability of a vortex ring
based on the axisymmetric Euler equation is
conducted by Choi \cite{79}, Cao, Qin, Zhan, and Zou \cite{72}, and
Choi and Jeong \cite{78}. In \cite{72}, they prove the orbital stability 
of a steady vortex ring introduced by Norbury \cite{51}.
Choi \cite{79} and Choi and Jeong \cite{78} prove that vortex rings with the core vortex distribution 
given by the Hill's vortex is orbitally stable. They also show estimates
for the size of the translational perturbation with respect to time.

Comparatively, a fewer number of research
\cite{21, 24, 81, 12, 82, 84} is available which study the 
LIE to investigate the motion of circular vortex filaments in detail. 
So far it is shown, in a mathematically rigorous context, that
circular vortex filaments and some closed vortex filaments are 
linearly stable under certain perturbations.
For instance, in Calini, Keith, and Lafortune \cite{82} and Ivey and Lafortune \cite{84},
they develop a technique and utilize it to prove that
closed vortex filaments are stable under the linearized LIE. 
As far as the author knows, there are no results proving the nonlinear
stability or instability of a circular vortex filament under the LIE.
We address this problem and prove stability estimates for a 
circular vortex filament under the LIE with some symmetry 
assumptions on the perturbation. The estimate we obtain 
admits a linear growth with respect to time in the axis direction, but we show that 
this is optimal.

\medskip

The contents of the rest of the paper are as follows.
In Section 2, we define function spaces and notations that are 
used throughout this paper. We also define compatibility conditions
and other necessary definitions before stating the main theorems.
In Section 3, we prove stability estimates for the arc-shaped filament.
We prove that an arc-shaped filament is stable in the 
Lyapunov sense, globaly in time, except in the \( \xi_{3}\)-direction.
Our estimates admit a linear growth with respect to \( t\) for the 
perturbation in the \( \xi_{3}\)-direction, and we also prove that 
this is optimal.
In Section 4, we utilize the results of Section 3 to prove 
stability estimates for a circular vortex filament. 
Under certain symmetry assumptions on the initial perturbation, 
we show that the initial value problem describing the 
motion of a perturbed circular vortex filament can be 
divided into segments, each being a solution of an
initial-boundary value problem describing the motion of a
perturbed arc-shaped filament. Hence, the results of 
Section 3 can be applied to each segment to obtain
estimates for the whole closed filament.
In Section 5, we give concluding remarks and compare our results to 
the analysis of a vortex ring.


\section{Function Spaces, Notations, and Main Theorem}
\setcounter{equation}{0}

We introduce some function spaces that will be used throughout this paper, and notations associated with the spaces.
Let \( I\subset \mathbf{R}\) be an open interval and 
\( \mathbf{T}=\mathbf{R}/[0,L]\). For \( \Omega = I \ \text{or} \ \mathbf{T}\),
a non-negative integer \( m\), and \( 1\leq p \leq \infty \), 
\( W^{m,p}(\Omega)\) 
is the Sobolev space 
containing all real-valued functions that have derivatives in the sense of distribution up to order \( m\) 
belonging to \( L^{p}(\Omega)\).
We set \( H^{m}(\Omega) := W^{m,2}(\Omega) \) as the 
Sobolev space equipped with the usual inner product, and 
set \( H^{1}_{0}(I) \) as the closure, with 
respect to the \( H^{1}\)-norm, of the set of smooth functions with compact support.
The norm in \( H^{m}(\Omega) \) is denoted by \( \| \cdot \|_{m} \) and we 
simply write \( \| \cdot \| \) 
for \( \|\cdot \|_{0} \). Otherwise, for a Banach space \( X\), the norm in \( X\) is written as \( \| \cdot \| _{X}\).
The inner product in \( L^{2}(\Omega)\) is denoted by \( (\cdot ,\cdot )\).

For \( 0<T \leq \infty \) and a Banach space \( X\), 
\( C^{m}([0,T];X) \)
( \( C^{m}\big( [0,\infty);X\big)\) when \( T= \infty \)),
denotes the space of functions that are \( m\) times continuously differentiable 
in \( t\) with respect to the norm of \( X\).
The space \( L^{\infty}\big( 0,\infty; X \big) \) denotes the space of functions 
that are bounded in \( t\) with respect to the norm of \( X\)

For any function space described above, we say that a vector valued function belongs to the function space 
if each of its components does.

\medskip

Next we define compatibility conditions for problem
(\ref{slant2}).
\begin{df}
(Compatibility conditions for {\rm (\ref{slant2})}).
For \( \mbox{\mathversion{bold}$x$}_{0}\in H^{2}(I_{L})\), we say that 
\( \mbox{\mathversion{bold}$x$}_{0}\) satisfies the \( 0\)-th order compatibility condition
for {\rm (\ref{slant2})} if
\begin{align*}
\mbox{\mathversion{bold}$x$}_{0s}(0)=\mbox{\mathversion{bold}$b$}, \quad 
\mbox{\mathversion{bold}$x$}_{0s}(L)=\mbox{\mathversion{bold}$e$}_{2},
\end{align*}
are satisfied.
For \( \mbox{\mathversion{bold}$x$}_{0}\in H^{4}(I_{L})\), we say that 
\( \mbox{\mathversion{bold}$x$}_{0}\) satisfies the \( 1\)-st order compatibility condition 
for {\rm (\ref{slant2})} if
\begin{align*}
\mbox{\mathversion{bold}$x$}_{0s}\times \mbox{\mathversion{bold}$x$}_{0sss}\big|_{s=0}
=
\mbox{\mathversion{bold}$x$}_{0s}\times \mbox{\mathversion{bold}$x$}_{0sss}\big|_{s=L}
=
\mbox{\mathversion{bold}$0$}
\end{align*}
are satisfied, where \( \big|_{s=0}\) and \( \big|_{s=L}\) denotes the traces at
\( s=0 \) and \( s=L\), respectively.
We also say that \( \mbox{\mathversion{bold}$x$}_{0}\in H^{4}(I_{L}) \) satisfies the 
compatibility conditions for {\rm (\ref{slant2})} up to order \( 1\) if
\( \mbox{\mathversion{bold}$x$}_{0}\) satisfies both the \( 0\)-th order and the 
\( 1\)-st order compatibility condition for {\rm (\ref{slant2})}.
\end{df}
We now introduce the assumptions we impose on the initial perturbation 
\( \mbox{\mathversion{bold}$\varphi$}_{0} \). Since the problem we consider is
\begin{align}
\left\{
\begin{array}{ll}
\mbox{\mathversion{bold}$x$}_{t} =\mbox{\mathversion{bold}$x$}_{s}\times 
\mbox{\mathversion{bold}$x$}_{ss}, & s\in I_{L}, \ t>0, \\[3mm]
\mbox{\mathversion{bold}$x$}(s,0)=\mbox{\mathversion{bold}$x$}^{R}_{0}(s)
+ \mbox{\mathversion{bold}$\varphi $}_{0}(s), & s\in I_{L}, \ t>0, \\[3mm]
\mbox{\mathversion{bold}$x$}_{s}(0,t)=\mbox{\mathversion{bold}$b$}, \ 
\mbox{\mathversion{bold}$x$}_{s}(L,t)=\mbox{\mathversion{bold}$e$}_{2}, & t>0,
\end{array}\right.
\label{slant3}
\end{align}
where
\begin{align*}
\mbox{\mathversion{bold}$x$}^{R}_{0}(s)={}^{t}\big( R\cos (\frac{s}{R}),
R\sin (\frac{s}{R}), 0 \big), \quad
\mbox{\mathversion{bold}$b$}={}^{t}(-\sin(\theta R), \cos(\theta R), 0),
\end{align*}
we arrive at assuming the following.
\begin{as}(Assumptions on \( \mbox{\mathversion{bold}$\varphi$}_{0}\)).
For the initial perturbation \( \mbox{\mathversion{bold}$\varphi$}_{0}\in H^{4}(I_{L})\),
we assume the following.
\begin{description}
\item[(A1)] \( | \mbox{\mathversion{bold}$x$}^{R}_{0s}(s)+
\mbox{\mathversion{bold}$\varphi$}_{0s}(s)| = 1 \) for all \( s\in I_{L}\). 

\item[(A2)] \( \mbox{\mathversion{bold}$x$}^{R}_{0}+\mbox{\mathversion{bold}$\varphi$}_{0} \)
satisfies the compatibility conditions for {\rm (\ref{slant2})} up to order \( 1\).
\item[(A3)] \( \mbox{\mathversion{bold}$e$}_{2}\cdot \mbox{\mathversion{bold}$\varphi$}_{0}(0)
= \mbox{\mathversion{bold}$b$}\cdot \mbox{\mathversion{bold}$\varphi$}_{0}(L)=0\). Here,
\( \cdot \) denotes the inner product in the three-dimensional Euclidean space.
\end{description}
\label{as}
\end{as}
For convenience, we will refer to the solution \( \mbox{\mathversion{bold}$x$}\) of 
(\ref{slant3}) as the perturbed filament, refer to 
\( \mbox{\mathversion{bold}$x$}^{R}_{0}
+ \mbox{\mathversion{bold}$\varphi $}_{0}\) as the perturbed initial filament,  
refer to \( \mbox{\mathversion{bold}$x$}^{R} \) as the arc-shaped filament,
and \( \mbox{\mathversion{bold}$x$}^{R}_{0}\) as the initial arc-shaped filament.
Assumption (A1) means that the initial perturbation is chosen so that it does not 
stretch the filament. Since the domain \( I_{L} \) of the problem (\ref{slant2}) is 
determined by the length of the initial filament, this assumption insures that both 
the arc-shaped filament \( \mbox{\mathversion{bold}$x$}^{R} \) and the perturbed filament
\( \mbox{\mathversion{bold}$x$}\) are defined on the same domain and is comparable. In other
words, the difference of the two, \( \mbox{\mathversion{bold}$\varphi $}(s,t)\), is
well-defined. Note that since the length of the arc-shaped filament is 
\( \theta R \), where \( \theta \) is the angle the lower and upper planes make, we have
\( L=\theta R \).
Assumption (A2) insures that problem (\ref{slant3}) has a
unique time-global solution which was obtained in \cite{23}.
Assumption (A3) insures that the end-points
\( \mbox{\mathversion{bold}$x$}(0,t) \) and \( \mbox{\mathversion{bold}$x$}(L,t)\)
of the perturbed filament stay on the 
lower and upper planes, respectively, for all \( t>0\). 
Indeed, 
substituting \( \mbox{\mathversion{bold}$x$}(s,t)= \mbox{\mathversion{bold}$x$}^{R}(s,t)
+ \mbox{\mathversion{bold}$\varphi $}(s,t) \) into (\ref{slant3}) shows that
\( \mbox{\mathversion{bold}$\varphi $}\) satsifies the following problem.
\begin{align}
\left\{
\begin{array}{ll}
\mbox{\mathversion{bold}$\varphi $}_{t}=
\mbox{\mathversion{bold}$\varphi $}_{s}\times \mbox{\mathversion{bold}$\varphi $}_{ss}
+
\mbox{\mathversion{bold}$x$}^{R}_{s}\times \mbox{\mathversion{bold}$\varphi $}_{ss}
+
\mbox{\mathversion{bold}$\varphi $}_{s}\times \mbox{\mathversion{bold}$x$}^{R}_{ss},
& s\in I_{\theta R}, t>0, \\[3mm]
\mbox{\mathversion{bold}$\varphi $}(s,0)=\mbox{\mathversion{bold}$\varphi $}_{0}(s),
& s\in I_{\theta R}, \\[3mm]
\mbox{\mathversion{bold}$\varphi $}_{s}(0)=\mbox{\mathversion{bold}$\varphi $}_{s}(\theta R)=
\mbox{\mathversion{bold}$0$}, & t>0.
\end{array}\right.
\label{pert}
\end{align}
Note that \( \mbox{\mathversion{bold}$x$}^{R}_{0s}(0)=
\mbox{\mathversion{bold}$e$}_{2}\), \( \mbox{\mathversion{bold}$x$}^{R}_{0s}(\theta R)=
\mbox{\mathversion{bold}$b$}\), and the fact that \( L=\theta R\) was used.
From (\ref{pert}), we can directly calculate to see that
\begin{align*}
(\mbox{\mathversion{bold}$e$}_{2}\cdot \mbox{\mathversion{bold}$\varphi $})_{t}
=\mbox{\mathversion{bold}$e$}_{2}\cdot \big( 
\mbox{\mathversion{bold}$x$}^{R}_{s}\times \mbox{\mathversion{bold}$\varphi $}_{ss}\big)\big|_{s=0}
=
\mbox{\mathversion{bold}$e$}_{2}\cdot \big( 
\mbox{\mathversion{bold}$e$}_{2}\times \mbox{\mathversion{bold}$\varphi $}_{ss}\big)\big|_{s=0}
=
0,
\end{align*}
where \( \mbox{\mathversion{bold}$x$}^{R}_{s}(0,t)=\mbox{\mathversion{bold}$e$}_{2} \) was 
substituted,
and similar calculations hold at \( s=\theta R\).
This proves that if the initial perturbation doesn't move the end-points of the 
arc-shaped filament away from their respective planes, 
then the end-points of the perturbed filament 
will stay on their respective planes for all time.

\medskip

Under these assumptions, we prove the following.
\begin{Th}
For any \( R>0\) and \( \theta \in (0,\pi )\), 
there exists \( C_{\ast}>0\) such that for any 
initial perturbation \( \mbox{\mathversion{bold}$\varphi $}_{0}
\in H^{4}(I_{\theta R})\) satisfying Assumption {\rm \ref{as}}, 
problem {\rm (\ref{slant3})} has a 
unique time-global solution \( \mbox{\mathversion{bold}$x $}(s,t) \) satisfying
\( |\mbox{\mathversion{bold}$x $}(s,t)|=1\) for all 
\( s\in I_{\theta R} \) and \( t>0 \),
\begin{align*}
\mbox{\mathversion{bold}$x$}\in \bigcap ^{2}_{j=0}C^{j}\big([0,\infty);H^{4-2j}(I_{\theta R})\big), \qquad and
\qquad 
\mbox{\mathversion{bold}$x$}_{s}\in L^{\infty}\big( 0,\infty; H^{3}(I_{\theta R})\big).
\end{align*}
Furthermore, 
\( \mbox{\mathversion{bold}$\varphi $}(s,t)=\mbox{\mathversion{bold}$x$}(s,t)-
\mbox{\mathversion{bold}$x$}^{R}(s,t)\)
 satisfies the following estimates.
In what follows, \( \varphi_{j}\) and \( \varphi_{0,j}\) \( (j=1,2,3)\) 
are the \( j\)-th component of \( \mbox{\mathversion{bold}$\varphi $}\) 
and \( \mbox{\mathversion{bold}$\varphi $}_{0}\) respectively. 
\begin{description}
\item[\quad (i)] For any \( t>0 \), 
\begin{align}
\| {}^{t}(\varphi_{1}(t),
\varphi_{2}(t))\| + 
\|\mbox{\mathversion{bold}$\varphi $}_{s}(t)\|_{1}\leq C_{\ast}
\| \mbox{\mathversion{bold}$\varphi $}_{0ss}\|
\label{est1}
\end{align}
holds. 
\item[\quad (ii)] For any \( t>0 \),
\begin{align}
\|\mbox{\mathversion{bold}$\varphi $}_{sss}(t)\|_{1} \leq 
C_{\ast}(\|\mbox{\mathversion{bold}$\varphi $}_{0ss}\|_{2} + \|\mbox{\mathversion{bold}$\varphi $}_{0ss}\|_{2}^{3})
\label{est2}
\end{align}
holds. 
\item[\quad (iii)] For any \( t>0 \),
\begin{align}
\| \varphi_{3}(t) \| 
\leq C_{\ast}(\|\varphi_{0,3}\|+\|\mbox{\mathversion{bold}$\varphi $}_{0ss}\|
+\|\mbox{\mathversion{bold}$\varphi $}_{0ss}\|^2t )
\label{est3}
\end{align}
holds.
\item[\quad (iv)] In particular, when \( \mbox{\mathversion{bold}$\varphi $}_{0}\) is a constant, then the estimate in 
{\rm (iii)} implies time-global boundedness for \( \varphi_{3}\).
When \( \mbox{\mathversion{bold}$\varphi $}_{0}\) is not a constant, any 
\( \mbox{\mathversion{bold}$\varphi $}_{0}\) satisfying the assumptions of the theorem necessarily satisfies
\( \| \mbox{\mathversion{bold}$\varphi $}_{0ss}\| > 0 \), and 
estimate {\rm (iii)} admits a linear growth with respect to \( t\).
\end{description}
Note that from the definition of the arc-shaped filament \( \mbox{\mathversion{bold}$x$}^{R} \), 
\( \mbox{\mathversion{bold}$\varphi $}\) and \( \mbox{\mathversion{bold}$\varphi $}_{s}\) belong to the same 
function space as \( \mbox{\mathversion{bold}$x$}\) and \( \mbox{\mathversion{bold}$x$}_{s}\), respectively.

\label{th1}
\end{Th}
Theorem \ref{th1} states that, in general, the arc-shaped filament is stable in the Lyapunov sense except in the 
direction in which the arc travels (which is the \( \xi_{3} \)-direction in our setting).
When \( \theta \in [\pi , 2\pi )\), we have the same result under 
an additional symmetry assumption on the initial perturbation. 
\begin{df}(Symmetry with respect to \( s=\frac{\theta R}{2}\)). 
Fix arbitrary \( R>0\) and \( \theta \in (0,2\pi )\). 
For \( \mbox{\mathversion{bold}$\varphi $}_{0}
={}^{t}(\varphi_{01},\varphi_{02},\varphi _{03})\in H^{2}(I_{\theta R})\), 
define \( \varphi _{0}^{r}\) and 
\( \varphi _{0}^{\theta } \) by
\begin{align*}
\begin{pmatrix}
\varphi_{0}^{r} \\[3mm]
\varphi _{0}^{\theta}
\end{pmatrix}
=
\begin{pmatrix}
\cos \frac{\theta }{2} & \sin \frac{\theta}{2} \\[3mm]
-\sin \frac{\theta}{2} & \cos \frac{\theta }{2}
\end{pmatrix}
\begin{pmatrix}
\varphi _{0,1}\\[3mm]
\varphi _{0,2}
\end{pmatrix}.
\end{align*}
\( \varphi^{r}_{0}\) and \( \varphi^{\theta}_{0}\) are the 
components of the two-dimensional vector
\( {}^{t}(\varphi_{01},\varphi_{02})\) expressed in terms of the basis
\( \{ \mbox{\mathversion{bold}$e$}^{r}, \mbox{\mathversion{bold}$e$}^{\theta}\}\), where
\begin{align*}
\mbox{\mathversion{bold}$e$}^{r}={}^{t}\big(\cos \frac{\theta }{2}, \sin \frac{\theta }{2}\big),
\qquad 
\mbox{\mathversion{bold}$e$}^{\theta}={}^{t}\big(-\sin \frac{\theta}{2},\cos \frac{\theta}{2}\big).
\end{align*}
We say that 
\( \mbox{\mathversion{bold}$\varphi $}_{0}\) is symmetric with respect to 
\( s=\frac{\theta R}{2}\) if the following are satisfied.
\begin{description}
\item[\quad (i)] \( \varphi ^{r}_{0}\) and \( \varphi_{0,3} \) are even functions with respect to \( s=\frac{\theta R}{2}\).
\item[\quad (ii)] \( \varphi ^{\theta}_{0}\) is an odd function with respect to \( s=\frac{\theta R}{2}\).
\end{description}
\label{sym}
\end{df}
Geometrically, if \( \mbox{\mathversion{bold}$\varphi $}_{0} \) is 
symmetric with respect to \( s=\frac{\theta R}{2}\), it means that 
the space curve \( \mbox{\mathversion{bold}$ x $}^{R}_{0}
+ \mbox{\mathversion{bold}$\varphi $}_{0}\) has reflection symmetry 
with respect to the plane perpendicular to 
\( \mbox{\mathversion{bold}$e$}^{\theta } \).
\begin{Th}
For any \( R>0 \) and \( \theta \in [\pi ,2\pi) \), there is a constant 
\( C_{\ast \ast}>0\) such that 
if the initial perturbation \( \mbox{\mathversion{bold}$\varphi $}_{0}
\in H^{4}(I_{\theta R}) \) satisfies Assumpstion {\rm \ref{as}} and is symmetric with respect to \( s=\frac{\theta R}{2}\),
then the conclusions of Theorem {\rm \ref{th1}} hold with \( C_{\ast}\) replaced with \( C_{\ast \ast} \).
\label{th2}
\end{Th}
Theorem \ref{th2} can be shown from Theorem \ref{th1} and utilizing the 
extra symmetry assumption.
Hence we focus on proving Theorem \ref{th1} and make a remark on Theorem \ref{th2} at the end.

The estimate for \( \varphi_{3}\) in Theorem \ref{th1} is optimal in the following sense.
\begin{Th}
There is a series of initial perturbations 
\( \{ \mbox{\mathversion{bold}$\phi $}_{n}\}_{n=1}^{\infty} \subset
H^{4}(I_{\theta R}) \) 
such that each \( \mbox{\mathversion{bold}$\phi $}_{n}\)
 satisfies Assumption {\rm \ref{as}} and 
the solution 
\( \mbox{\mathversion{bold}$ x $}_{n}\) of problem {\rm (\ref{slant3})} with 
initial datum \( \mbox{\mathversion{bold}$ x $}^{R}_{0} + 
\mbox{\mathversion{bold}$ \phi$}_{n} \) satisfies
\begin{align}
| x_{n,3}(s,t)-x^{R}_{3}(s,t)| = \left( \frac{2\pi n}{R\theta}\right)t
\label{opt1}
\end{align}
for all \( s\in I_{\theta R}\) and \( t\geq 0\). 
Here, \( x_{n,3}\) and \(x^{R}_{3}\) are the third component of 
\( \mbox{\mathversion{bold}$ x $}_{n}\) and
\( \mbox{\mathversion{bold}$ x $}^{R} \), respectively.

\label{thopt}

\end{Th}
Theorem \ref{thopt} shows that statement (iii) in Theorem \ref{th1} is 
optimal. Moreover, (\ref{opt1}) shows that
the rate at which the perturbed filament can depart from the 
arc-shaped filament is unbounded.


\section{Proof of Theorem \ref{th1}, \ref{th2}, and \ref{thopt}}
\setcounter{equation}{0}

We prove Theorem \ref{th1}, \ref{th2}, and \ref{thopt} in this section.
Since Theore \ref{th2} follows immediately from 
Theorem \ref{th1}, we prove Theorem \ref{th1} and Theorem \ref{thopt} first,
and give a remark on the proof of Theorem \ref{th2} at the 
end of the section.

\subsection{Proof of Theorem \ref{th1}}

Since the existence of the time-global solution \( \mbox{\mathversion{bold}$x$}(s,t)\) for problem (\ref{slant3}) is proved in
\cite{23}, we prove the estimates (\ref{est1}), (\ref{est2}), (\ref{est3}).
We first prove the following.
\begin{pr}
For \( \mbox{\mathversion{bold}$\varphi $}(s,t) \) given in Theorem 
{\rm \ref{th1}}, set
\begin{align*}
E(\mbox{\mathversion{bold}$\varphi $}(t)):=
\| \mbox{\mathversion{bold}$\varphi $}_{ss}(t)\|^{2}-\frac{1}{R^{2}}\|\mbox{\mathversion{bold}$\varphi $}_{s}(t)\|^{2}.
\end{align*}
Then, \( E(\mbox{\mathversion{bold}$\varphi $}(t))=
E(\mbox{\mathversion{bold}$\varphi $}_{0})\)
for all \( t>0\).
\label{cons}
\end{pr}
{\it Proof.} From (\ref{pert}), we have
\begin{align*}
\frac{{\rm d}}{{\rm d}t}\|\mbox{\mathversion{bold}$\varphi $}_{s}\|^{2}
=
2(\mbox{\mathversion{bold}$\varphi $}_{s},\mbox{\mathversion{bold}$\varphi $}_{st})
&=
2(\mbox{\mathversion{bold}$\varphi $}_{s}, \mbox{\mathversion{bold}$\varphi $}_{s}\times 
\mbox{\mathversion{bold}$\varphi $}_{sss}) 
+ 
2( \mbox{\mathversion{bold}$\varphi $}_{s}, 
\mbox{\mathversion{bold}$x$}^{R}_{s}\times \mbox{\mathversion{bold}$\varphi $}_{sss})
+
2( \mbox{\mathversion{bold}$\varphi $}_{s}, \mbox{\mathversion{bold}$\varphi $}_{s}\times
\mbox{\mathversion{bold}$x$}^{R}_{sss})\\[3mm]
&=
-2(\mbox{\mathversion{bold}$\varphi $}_{s},\mbox{\mathversion{bold}$x$}^{R}_{ss}\times
\mbox{\mathversion{bold}$\varphi $}_{ss})
\end{align*}
where integration by parts was used. Furthermore,
\begin{align*}
\frac{{\rm d}}{{\rm d}t}\|\mbox{\mathversion{bold}$\varphi $}_{ss}\|^{2}
=
2(\mbox{\mathversion{bold}$\varphi $}_{ss},\mbox{\mathversion{bold}$\varphi $}_{sst})
=
-2(\mbox{\mathversion{bold}$\varphi $}_{sss}, \mbox{\mathversion{bold}$\varphi $}_{st})
&=-2(\mbox{\mathversion{bold}$\varphi $}_{sss},\mbox{\mathversion{bold}$\varphi $}_{s}\times
\mbox{\mathversion{bold}$x$}^{R}_{sss})\\[3mm]
&= 2(\mbox{\mathversion{bold}$\varphi $}_{ss},\mbox{\mathversion{bold}$\varphi $}_{s}\times
\partial^{4}_{s}\mbox{\mathversion{bold}$x$}^{R} )\\[3mm]
&= -\frac{2}{R^{2}}(\mbox{\mathversion{bold}$\varphi $}_{ss},
\mbox{\mathversion{bold}$\varphi $}_{s}\times \mbox{\mathversion{bold}$x$}^{R}_{ss})
\end{align*}
holds, where integration by parts was used and 
\( \partial^{4}_{s}\mbox{\mathversion{bold}$x$}^{R} =
-\frac{1}{R^{2}}\mbox{\mathversion{bold}$x$}^{R}_{ss}\) was substituted. Combining the 
two equalities shows that \( \frac{{\rm d}}{{\rm d}t}E(\mbox{\mathversion{bold}$\varphi $}(t))
=0 \), and proves the proposition. \hfill \( \Box \).

\bigskip

The above conserved quantity allows us to prove the following
alternative for \( \mbox{\mathversion{bold}$\varphi $}_{ss}\).
\begin{lm}
\( \mbox{\mathversion{bold}$\varphi $}(s,t)\) given in Theorem {\rm \ref{th1}} 
satisfies the
following.
\begin{description}
\item[ (1)] If \( \mbox{\mathversion{bold}$\varphi $}_{0}\) is a constant vector, 
then \( \mbox{\mathversion{bold}$\varphi $}(s,t)=\mbox{\mathversion{bold}$\varphi $}_{0}(s)\)
for all \( s\in I_{\theta R}\) and \( t>0\). In particular, 
\( \| \mbox{\mathversion{bold}$\varphi $}_{ss}(t)\|=0 \) for all \( t>0 \).
\item[ (2)] If \( \mbox{\mathversion{bold}$\varphi $}_{0}\) is not a 
constant vector, then \( \| \mbox{\mathversion{bold}$\varphi $}_{0ss}\|>0 \) and
\begin{align*}
\| \mbox{\mathversion{bold}$\varphi $}_{ss}(t)\|\geq
\big( 1-\frac{\theta ^{2}}{\pi^{2}}\big)^{\frac{1}{2}}
\| \mbox{\mathversion{bold}$\varphi $}_{0ss}\|
\end{align*}
for all \( t>0\).
\end{description}
In other words, \( \| \mbox{\mathversion{bold}$\varphi $}_{ss}(t)\| \) is either 
\( 0 \) or strictly positive for all \( t>0\). This shows that 
the arc-shaped filament is not asymptotically stable under any non-zero 
perturbation. 
\label{nondecay}
\end{lm}
{\it Proof}. If \( \mbox{\mathversion{bold}$\varphi $}_{0}\) is a constant vector,
then \( \mbox{\mathversion{bold}$\varphi $}_{0}\) must have the form
\( \mbox{\mathversion{bold}$\varphi $}_{0}={}^{t}(0,0, c) \) for some constant \( c\)
because of Assumption \ref{as}.
Then, 
direct calculation yields that 
\( \mbox{\mathversion{bold}$x$}(s,t)=\mbox{\mathversion{bold}$x$}^{R}(s,t) + 
\mbox{\mathversion{bold}$\varphi $}_{0}(s) \) is the unique solution of
(\ref{slant3}) and (1) of the lemma holds.

Suppose \( \mbox{\mathversion{bold}$\varphi $}_{0}\) is not a constant.
From (A2) of Assumption \ref{as}, we obtain 
\( \mbox{\mathversion{bold}$\varphi $}_{0s}\big|_{s=0,\theta R} = 
\mbox{\mathversion{bold}$0$}\), and we have
\begin{align*}
\| \mbox{\mathversion{bold}$\varphi $}_{ss}(t)\|^{2}
\geq E(\mbox{\mathversion{bold}$\varphi $}(t)) 
= E(\mbox{\mathversion{bold}$\varphi $}_{0})
\geq
\big( 1-\frac{\theta^{2}}{\pi^{2}}\big)\| 
\mbox{\mathversion{bold}$\varphi $}_{0ss}\|^{2}.
\end{align*}
Here, the Poincar\'e inequality 
\begin{align*}
\| \mbox{\mathversion{bold}$\varphi $}_{0s}\| \leq
\frac{\theta R}{\pi} \| \mbox{\mathversion{bold}$\varphi $}_{0ss}\|
\end{align*}
with the sharp constant \( \frac{\theta R}{\pi }\),
which is the smallest eigenvalue of the Laplacian in \( H^{1}_{0}(I_{\theta R})\), was applied.
Note that \( \theta \in (0,\pi) \) insures that 
\( 1-\frac{\theta^{2}}{\pi^{2}} >0 \). 
Finally, we show that \( \| \mbox{\mathversion{bold}$\varphi $}_{0ss}\| >0 \)
by contradiction. Suppose \( \| \mbox{\mathversion{bold}$\varphi $}_{0ss}\| =0\). 
Then, by the Poincar\'e inequality we see that 
\( \| \mbox{\mathversion{bold}$\varphi $}_{0s}\| = 0\). This in return means that 
\( \mbox{\mathversion{bold}$\varphi $}_{0}\) is a constant, which is a 
contradiction. Hence, the lemma is proven. \hfill \( \Box. \)

\medskip

Since \( \mbox{\mathversion{bold}$\varphi $}_{s}|_{s=0,\theta R}=0\), 
we have
\begin{align*}
\| \mbox{\mathversion{bold}$\varphi $}_{s}(t) \|^{2}
\leq \frac{\theta ^{2} R^{2}}{\pi ^{2}}\| \mbox{\mathversion{bold}$\varphi $}_{ss}(t)\|^{2},
\end{align*}
for all \( t>0\). 
Hence, 
\begin{align*}
E(\mbox{\mathversion{bold}$\varphi $}(t)) = 
\| \mbox{\mathversion{bold}$\varphi $}_{ss}(t)\|^{2}-
\frac{1}{R^{2}}\|\mbox{\mathversion{bold}$\varphi $}_{s}(t)\|^{2}
\geq
\big( 1-\frac{\theta^{2}}{\pi^{2}}\big)\| \mbox{\mathversion{bold}$\varphi $}_{ss}\|^{2}
\end{align*}
and since we are assuming \( \theta \in (0, \pi )\), we have for all \( t>0\)
\begin{align*}
\| \mbox{\mathversion{bold}$\varphi $}_{ss}(t)\|^{2} \leq
\big(1-\frac{\theta^{2}}{\pi ^{2}}\big)^{-1}E(\mbox{\mathversion{bold}$\varphi $}_{0})
\leq
\big(1-\frac{\theta^{2}}{\pi^{2}}\big)^{-1}\|\mbox{\mathversion{bold}$\varphi $}_{0ss}\|^{2},
\end{align*}
which also yields
\begin{align*}
\| \mbox{\mathversion{bold}$\varphi $}_{s}(t)\|^{2}\leq 
\frac{\theta^{2}R^{2}}{\pi ^{2}}\big( 1-\frac{\theta^{2}}{\pi^{2}}\big)^{-1}\| \mbox{\mathversion{bold}$\varphi $}_{0ss}\|^{2}.
\end{align*}
Hence, we have proven
\begin{lm}
\( \mbox{\mathversion{bold}$\varphi $}(s,t)\) given in Theorem {\rm \ref{th1}}
 satisfies
\begin{align*}
\| \mbox{\mathversion{bold}$\varphi $}_{s}(t)\|_{1} \leq 
\max\{ 1, \frac{\theta R}{\pi}\} \big( 1-\frac{\theta ^{2}}{\pi^{2}}\big)^{-\frac{1}{2}}
\| \mbox{\mathversion{bold}$\varphi $}_{0ss}\|
\end{align*}
for all \( t>0\). We denote \( C_{0}:= \max\{ 1, \frac{\theta R}{\pi}\} \big( 1-\frac{\theta ^{2}}{\pi^{2}}\big)^{-\frac{1}{2}} \).
\label{basic}
\end{lm}

Next we show the following higher-order estimate.
\begin{lm}(Higher-order estimate).
There exists \( C>0\) depending only on \( \theta \) and \( R\) such that
\( \mbox{\mathversion{bold}$\varphi $}(s,t)\) given in Theorem {\rm \ref{th1}}
 satisfies
\begin{align*}
\| \mbox{\mathversion{bold}$\varphi $}_{sss}(t)\|_{1}
\leq
C( \|\mbox{\mathversion{bold}$\varphi $}_{0ss}\|_{2} + 
\|\mbox{\mathversion{bold}$\varphi $}_{0ss}\|^{3}_{2})
\end{align*}
for all \( t>0\).
\end{lm}
{\it Proof.}
We make use of conserved quantities for problem (\ref{slant3}), which was also utilized in \cite{23}. 
Set \( \mbox{\mathversion{bold}$v$}:=\mbox{\mathversion{bold}$x$}_{s}\). Then, we see that 
\( \mbox{\mathversion{bold}$v$}\) satisfies
\begin{align}
\left\{
\begin{array}{ll}
\mbox{\mathversion{bold}$v$}_{t}=\mbox{\mathversion{bold}$v$}\times \mbox{\mathversion{bold}$v$}_{ss},
& s\in I_{\theta R}, \ t>0, \\[3mm]
\mbox{\mathversion{bold}$v$}(s,0)=\mbox{\mathversion{bold}$x$}^{R}_{0s}(s)+\mbox{\mathversion{bold}$\varphi$}_{0s}(s), 
& s\in I_{\theta R}, \\[3mm]
\mbox{\mathversion{bold}$v$}(0,t)=\mbox{\mathversion{bold}$b$},
 \ \mbox{\mathversion{bold}$v$}(\theta R, t)=\mbox{\mathversion{bold}$e$}_{2}, 
 & t>0.
\end{array}\right.
\label{veq}
\end{align}
We immediately see that \( \frac{{\rm d}}{{\rm d}t}|\mbox{\mathversion{bold}$v$}|^{2} = 0 \), and 
from (A1) of Assumption \ref{as} we have
\begin{align*}
|\mbox{\mathversion{bold}$x$}^{R}_{s}(s,t)+\mbox{\mathversion{bold}$\varphi$}_{s}(s,t)|^{2}
=
|\mbox{\mathversion{bold}$x$}^{R}_{0s}(s)+\mbox{\mathversion{bold}$\varphi$}_{0s}(s)|^{2}
=
1.
\end{align*}
Since \( |\mbox{\mathversion{bold}$x$}^{R}_{s}(s,t)| = 1 \), we have
\begin{align}
2\mbox{\mathversion{bold}$x$}^{R}_{s}(s,t) \cdot \mbox{\mathversion{bold}$\varphi$}_{s}(s,t)
=
-|\mbox{\mathversion{bold}$\varphi$}_{s}(s,t)|^{2}
\label{nostretch}
\end{align}
for all \( s\in I_{\theta R}\) and \( t>0\).
We also have the following conserved quantities.
\begin{align*}
&\frac{{\rm d}}{{\rm d}t}\bigg\{
\| \mbox{\mathversion{bold}$v$}_{ss}\|^{2}-\frac{5}{4}
\big\| |\mbox{\mathversion{bold}$v$}_{s}|^{2}\big\|^{2}
\bigg\}
= 0, \\[3mm]
&\frac{{\rm d}}{{\rm d}t}\bigg\{
\|\mbox{\mathversion{bold}$v$}_{sss}\|^{2}
-\frac{7}{2}\big\| |\mbox{\mathversion{bold}$v$}_{s}|
|\mbox{\mathversion{bold}$v$}_{ss}|\big\|^{2}-
14\| \mbox{\mathversion{bold}$v$}_{s}\cdot \mbox{\mathversion{bold}$v$}_{ss}\|^{2}
+
\frac{21}{8}\big\|
|\mbox{\mathversion{bold}$v$}_{s}|^{3}\big\| ^{2}
\bigg\}= 0.
\end{align*}
The above can be checked by direct calculation. First, we set
\begin{align*}
E_{1}(\mbox{\mathversion{bold}$v$}(t)):=
\| \mbox{\mathversion{bold}$v$}_{ss}(t)\|^{2}-\frac{5}{4}\big\| |\mbox{\mathversion{bold}$v$}_{s}(t)|^{2}\big\|^{2}.
\end{align*}
For \( t\geq 0 \), we substitute \( \mbox{\mathversion{bold}$v$}=\mbox{\mathversion{bold}$x$}^{R}_{s}+
\mbox{\mathversion{bold}$\varphi$}_{s}\) and decompose as follows.
\begin{align}
E_{1}(\mbox{\mathversion{bold}$v$}(t)) = E_{1}(\mbox{\mathversion{bold}$x$}^{R}_{s}(t)) + 
E_{1}(\mbox{\mathversion{bold}$\varphi$}_{s}(t)) + R_{1}(\mbox{\mathversion{bold}$x$}^{R}_{s}(t),
\mbox{\mathversion{bold}$\varphi$}_{s}(t)),
\label{decom1}
\end{align}
where \( R_{1}(\mbox{\mathversion{bold}$x$}^{R}_{s}(t),
\mbox{\mathversion{bold}$\varphi$}_{s}(t)) \) is given by
\begin{align*}
R_{1}(\mbox{\mathversion{bold}$x$}^{R}_{s}(t),
\mbox{\mathversion{bold}$\varphi$}_{s}(t)) 
&=
2(\mbox{\mathversion{bold}$x$}^{R}_{sss}, \mbox{\mathversion{bold}$\varphi$}_{sss})
-5(|\mbox{\mathversion{bold}$x$}^{R}_{ss}|^{2}\mbox{\mathversion{bold}$x$}^{R}_{ss},\mbox{\mathversion{bold}$\varphi$}_{ss})
-5((\mbox{\mathversion{bold}$x$}^{R}_{ss}\cdot\mbox{\mathversion{bold}$\varphi$}_{ss})
  \mbox{\mathversion{bold}$x$}^{R}_{ss},\mbox{\mathversion{bold}$\varphi$}_{ss})\\[3mm]
& \qquad  -
\frac{5}{2}(|\mbox{\mathversion{bold}$x$}^{R}_{ss}|^{2}\mbox{\mathversion{bold}$\varphi$}_{ss},\mbox{\mathversion{bold}$\varphi$}_{ss})
-
5(|\mbox{\mathversion{bold}$\varphi$}_{ss}|^{2}\mbox{\mathversion{bold}$x$}^{R}_{ss},\mbox{\mathversion{bold}$\varphi$}_{ss}).
\end{align*}
The terms that are linear with respect to \( \mbox{\mathversion{bold}$\varphi$} \) can be estimated as follows. First we have
\begin{align*}
2(\mbox{\mathversion{bold}$x$}^{R}_{sss},\mbox{\mathversion{bold}$\varphi$}_{sss})
=
-\frac{2}{R^{2}}(\mbox{\mathversion{bold}$x$}^{R}_{s},\mbox{\mathversion{bold}$\varphi$}_{sss})
&=
\frac{2}{R^{2}}(\mbox{\mathversion{bold}$x$}^{R}_{ss},\mbox{\mathversion{bold}$\varphi$}_{ss})
-\frac{2}{R^{2}}\big[ \mbox{\mathversion{bold}$x$}^{R}_{s}\cdot \mbox{\mathversion{bold}$\varphi$}_{ss}\big]^{\theta R}_{s=0}\\[3mm]
&=
-\frac{2}{R^{2}}(\mbox{\mathversion{bold}$x$}^{R}_{sss},\mbox{\mathversion{bold}$\varphi$}_{s})
-\frac{2}{R^{2}}\big[ \mbox{\mathversion{bold}$x$}^{R}_{s}\cdot \mbox{\mathversion{bold}$\varphi$}_{ss}\big]^{\theta R}_{s=0}\\[3mm]
&=
\frac{2}{R^{4}}(\mbox{\mathversion{bold}$x$}^{R}_{s},\mbox{\mathversion{bold}$\varphi$}_{s})
-\frac{2}{R^{2}}\big[ \mbox{\mathversion{bold}$x$}^{R}_{s}\cdot \mbox{\mathversion{bold}$\varphi$}_{ss}\big]^{\theta R}_{s=0}
\end{align*}
where integration by parts, \( \mbox{\mathversion{bold}$\varphi$}_{s}|_{s=0,L}=\mbox{\mathversion{bold}$0$}\), and
\( \mbox{\mathversion{bold}$x$}^{R}_{sss}=-\frac{1}{R^{2}}\mbox{\mathversion{bold}$x$}^{R}_{s}\) was used.
Integrating equation (\ref{nostretch}) with respect to \( s\) shows that
 \( (\mbox{\mathversion{bold}$x$}^{R}_{s},\mbox{\mathversion{bold}$\varphi$}_{s})
= -\frac{1}{2}\|\mbox{\mathversion{bold}$\varphi$}_{s}\|^{2}\).
Furthermore, differentiating equation (\ref{nostretch}) with respect to \( s\) and taking the trace at \( s=0,\theta R\) yields
\( \mbox{\mathversion{bold}$x$}^{R}_{s}\cdot \mbox{\mathversion{bold}$\varphi$}_{ss}|_{s=0,\theta R}=0\).
Hence we have
\begin{align*}
2(\mbox{\mathversion{bold}$x$}^{R}_{sss},\mbox{\mathversion{bold}$\varphi$}_{sss})
=
-\frac{1}{R^{4}}\|\mbox{\mathversion{bold}$\varphi$}_{s}\|^{2}.
\end{align*}
Similarly, we have
\begin{align*}
-5(|\mbox{\mathversion{bold}$x$}^{R}_{ss}|^{2}\mbox{\mathversion{bold}$x$}^{R}_{ss},\mbox{\mathversion{bold}$\varphi$}_{ss})
=
-\frac{5}{R^{2}}(\mbox{\mathversion{bold}$x$}^{R}_{ss},\mbox{\mathversion{bold}$\varphi$}_{ss})
=
\frac{5}{2R^{4}}\|\mbox{\mathversion{bold}$\varphi$}_{s}\|^{2},
\end{align*}
where \( |\mbox{\mathversion{bold}$x$}^{R}_{ss}|=\frac{1}{R}\) was used.
From these equalities, we have  
\begin{align}
| R_{1}(\mbox{\mathversion{bold}$x$}^{R}_{s}(t),
\mbox{\mathversion{bold}$\varphi$}_{s}(t)) |
\leq 
C\big( \|\mbox{\mathversion{bold}$\varphi$}_{s}(t)\|^{2}_{1} + 
\big\| |\mbox{\mathversion{bold}$\varphi$}_{ss}(t)|^{3/2}\big\|^{2} \big),
\label{H3}
\end{align}
where \( C>0\) depends only on \( R\).
In addition, we have \( \mbox{\mathversion{bold}$x$}^{R}_{s}(s,t)=\mbox{\mathversion{bold}$x$}^{R}_{0s}(s)\)
from the definition of \( \mbox{\mathversion{bold}$x$}^{R}\). This, along with the conservation of 
\( E_{1}(\mbox{\mathversion{bold}$v$})\) and the decomposition (\ref{decom1}), yields 
\begin{align*}
E_{1}(\mbox{\mathversion{bold}$\varphi$}_{s}(t)) 
= E_{1}(\mbox{\mathversion{bold}$\varphi$}_{0s})
-
R_{1}(\mbox{\mathversion{bold}$x$}^{R}_{0s},\mbox{\mathversion{bold}$\varphi$}_{s})
+
R_{1}(\mbox{\mathversion{bold}$x$}^{R}_{0s},\mbox{\mathversion{bold}$\varphi$}_{0s}). \nonumber
\end{align*}
From the interpolation inequality, Lemma \ref{basic}, and inequality (\ref{H3}), the above equation for
\( E_{1}(\mbox{\mathversion{bold}$\varphi$}_{s}(t))\) shows that
\begin{align*}
\| \mbox{\mathversion{bold}$\varphi$}_{sss}(t)\|^{2}
\leq
C\big( \|\mbox{\mathversion{bold}$\varphi$}_{0ss}\|^{2}_{1} + \|\mbox{\mathversion{bold}$\varphi$}_{0ss}\|^{4}_{1}\big)
\end{align*}
holds for all \( t>0\), where \( C>0\) depends on \( R\) and \( \theta \). 
Similarly, if we set
\begin{align*}
E_{2}(\mbox{\mathversion{bold}$v$}(t))
=
\|\mbox{\mathversion{bold}$v$}_{sss}(t)\|^{2}
-\frac{7}{2}\big\| |\mbox{\mathversion{bold}$v$}_{s}(t)|
|\mbox{\mathversion{bold}$v$}_{ss}(t)|\big\|^{2}-
14\| \mbox{\mathversion{bold}$v$}_{s}(t)\cdot \mbox{\mathversion{bold}$v$}_{ss}(t)\|^{2}
+
\frac{21}{8}\big\|
|\mbox{\mathversion{bold}$v$}_{s}(t)|^{3}\big\|^{2},
\end{align*}
and proceed in the same way as \( E_{1}(\mbox{\mathversion{bold}$v$}(t))\), we obtain
\begin{align*}
\|\partial^{4}_{s}\mbox{\mathversion{bold}$\varphi$}(t)\|^{2}
\leq
C\big( \|\mbox{\mathversion{bold}$\varphi$}_{0ss}\|^{2}_{2}+\|\mbox{\mathversion{bold}$\varphi$}_{0ss}\|^{6}_{2}\big)
\end{align*}
and this finishes the proof. \hfill \( \Box \).

\bigskip

Finally, we derive estimates for \( \mbox{\mathversion{bold}$\varphi$}\) itself.
\begin{lm}
\( \mbox{\mathversion{bold}$\varphi$}(s,t)\) given in Theorem {\rm \ref{th1}}
 satisfies
\begin{align*}
\|{}^{t}(\varphi_{1}(t),\varphi_{2}(t))\|
\leq
C\|\mbox{\mathversion{bold}$\varphi$}_{0ss}\|
\end{align*}
for all \( t>0\), and 
\begin{align*}
\|\varphi_{3}(t)\| \leq C
\big( \| \varphi_{0,3}\|+\|\mbox{\mathversion{bold}$\varphi$}_{0ss}\|
+\|\mbox{\mathversion{bold}$\varphi$}_{0ss}\|^{2}t \big)
\end{align*}
for all \( t>0\).
Here, \( C>0\) depends only on \( \theta \) and \( R\).
\label{raw}
\end{lm}
{\it Proof}. We make use of the reflection symmetry of the arc-shaped filament and the invariance of the equation with 
respect to this symmetry. More specifically, for the perturbed initial filament \( \mbox{\mathversion{bold}$x$}_{0}\)
\( (=\mbox{\mathversion{bold}$x$}^{R}_{0}+\mbox{\mathversion{bold}$\varphi$}_{0})\),
we define an extended filament \( \tilde{\mbox{\mathversion{bold}$x$}}_{0}\) by
\begin{align*}
\tilde{\mbox{\mathversion{bold}$x$}}_{0}(s)
=
\left\{
\begin{array}{ll}
\mbox{\mathversion{bold}$x$}_{0}(s), & s\in [0,\theta R), \\[3mm]
{}^{t}(x_{0,1}(-s),-x_{0,2}(-s),x_{0,3}(-s)), & s\in (-\theta R, 0),
\end{array}\right.
\end{align*}
which is an extension made by reflecting the filament with respect to the plane 
\( \xi_{2}=0\). We further define the operator \( {\rm T}\) for vector valued functions 
\( \mbox{\mathversion{bold}$y$}(s)\) defined for \( s\in (-\theta R, \theta R)\) by
\begin{align*}
({\rm T}\mbox{\mathversion{bold}$y$})(s) = 
{}^{t}(y_{1}(-s),-y_{2}(-s),y_{3}(-s)).
\end{align*}
Then, from (A2) and (A3) of Assumption \ref{as}, we see
that \( \tilde{\mbox{\mathversion{bold}$x$}}_{0}\in H^{4}\big( (-\theta R,\theta R)\big)\)
and
\begin{align*}
{\rm T}\tilde{\mbox{\mathversion{bold}$x$}}_{0}= \tilde{\mbox{\mathversion{bold}$x$}}_{0}
\end{align*}
from the definition of \( \tilde{\mbox{\mathversion{bold}$x$}}_{0}\).
Moreover, we have
\begin{align*}
{\rm T}(\mbox{\mathversion{bold}$y$}_{s}\times \mbox{\mathversion{bold}$y$}_{ss})
&=
({\rm T}\mbox{\mathversion{bold}$y$})_{s}\times ({\rm T}\mbox{\mathversion{bold}$y$})_{ss}.
\end{align*}
Hence, if we consider the auxiliary problem
\begin{align}
\left\{
\begin{array}{ll}
\mbox{\mathversion{bold}$w$}_{t}=\mbox{\mathversion{bold}$w$}_{s}\times \mbox{\mathversion{bold}$w$}_{ss},
& s\in \tilde{I}, t>0, \\[3mm]
\mbox{\mathversion{bold}$w$}(s,0)=\tilde{\mbox{\mathversion{bold}$x$}}_{0}(s), & s\in \tilde{I}, \\[3mm]
\mbox{\mathversion{bold}$w$}_{s}(-\theta R, t)= \tilde{\mbox{\mathversion{bold}$b$}}, \ 
\mbox{\mathversion{bold}$w$}_{s}(\theta R,t)= \mbox{\mathversion{bold}$b$}, & t>0,
\end{array}\right.
\label{aux}
\end{align}
where \( \tilde{\mbox{\mathversion{bold}$b$}}={}^{t}(b_{1},-b_{2},0)\) and
\( \tilde{I}=(-\theta R, \theta R)\), we see that 
if \( \mbox{\mathversion{bold}$w$}\) is a solution, so is \( {\rm T}\mbox{\mathversion{bold}$w$}\).
The unique solvability of (\ref{aux}) follows from \cite{23} and in particular, 
the uniqueness of the solution implies \( 
{\rm T}\mbox{\mathversion{bold}$w$} = \mbox{\mathversion{bold}$w$}\).
In other words, the reflection symmetry is preserved throughout the motion.
Since the arc-shaped filament \( \mbox{\mathversion{bold}$x$}^{R}(s,t)\) can also be extended 
by reflection with preservation of smoothness, the extended arc-shaped filament
\( \tilde{\mbox{\mathversion{bold}$x$}}^{R} \) satisfies
\( {\rm T}\tilde{\mbox{\mathversion{bold}$x$}}^{R} = \tilde{\mbox{\mathversion{bold}$x$}}^{R}\) and 
hence,
\( \tilde{\mbox{\mathversion{bold}$\varphi$}}:= \mbox{\mathversion{bold}$w$}-
\tilde{\mbox{\mathversion{bold}$x$}}^{R}\) also satisfies 
\( {\rm T}\tilde{\mbox{\mathversion{bold}$\varphi$}} = \tilde{\mbox{\mathversion{bold}$\varphi$}}\).
Finally, since \( {\rm T}\mbox{\mathversion{bold}$w$} = \mbox{\mathversion{bold}$w$}\), we see that 
\( \mbox{\mathversion{bold}$w$}\big|_{I_{\theta R}}\), the restriction of \( \mbox{\mathversion{bold}$w$}\) 
to \( I_{\theta R}\), is a solution of 
(\ref{slant3}) and again by uniqueness, \( \mbox{\mathversion{bold}$w$}\big|_{I_{\theta R}} = \mbox{\mathversion{bold}$x$}\),
and as a consequence, \( \tilde{\mbox{\mathversion{bold}$\varphi$}}\big|_{I_{\theta R}} = 
\mbox{\mathversion{bold}$\varphi$}\).

From \( {\rm T}\tilde{\mbox{\mathversion{bold}$\varphi$}} = \tilde{\mbox{\mathversion{bold}$\varphi$}}\), 
we know that \( \tilde{\varphi}_{2}(s,t)\) is an odd function with respect to \( s=0 \), thus,
\begin{align*}
\int^{\theta R}_{-\theta R}\tilde{\varphi}_{2}(s,t){\rm d}s = 0
\end{align*}
for all \( t>0\). Hence from the Poincar\'e inequality and 
Lemma \ref{basic}, we have
\begin{align*}
\| \tilde{\varphi}_{2}(t)\|_{L^{2}(\tilde{I})}
 \leq \frac{\theta R}{\pi }\|\tilde{\varphi}_{2s}(t)\|_{L^{2}(\tilde{I})}
= \frac{2\theta R}{\pi}\|\varphi_{2s}(t)\|
\leq
\frac{2\theta R}{\pi }C_{0}\|\mbox{\mathversion{bold}$\varphi$}_{0ss}\|.
\end{align*}
Here, \( C_{0}>0\) is the constant obtained in Lemma \ref{basic}.
Then, the estimate for \( \varphi_{2}\) follows from
\( \| \tilde{\varphi}_{2}(t)\|_{L^{2}(\tilde{I})} =
2\| \varphi_{2}(t)\| \). Similarly, we can extend the filament by reflection 
at the upper plane, rotate the filament, and apply the same arguments to 
obtain 
\begin{align*}
\| \varphi^{b}(t)\| \leq \frac{\theta R}{\pi }C_{0}\|\mbox{\mathversion{bold}$\varphi$}_{0ss}\|,
\end{align*}
for all \( t>0\), where \( \varphi^{b}= \mbox{\mathversion{bold}$\varphi$}\cdot 
\mbox{\mathversion{bold}$b$}\). 
From
\begin{align*}
\varphi_{1}=-\frac{1}{\sin \theta }\varphi^{b} + \frac{\cos \theta}{\sin \theta}
\varphi_{2},
\end{align*}
we obtain the estimate for \( \varphi_{1}\).

Finally, we estimate \( \varphi_{3}\). From \ref{pert}, we see that \( \varphi_{3}\)
satisfies
\begin{align}
\varphi_{3t}=\varphi_{1s}\varphi_{2ss}-\varphi_{2s}\varphi_{1ss} -
\varphi_{2ss}\sin(\frac{s}{R}) - \varphi_{1ss}\cos(\frac{s}{R})
+
\frac{1}{R}\varphi_{2s}\cos (\frac{s}{R}) - \frac{1}{R}\varphi_{1s}\sin(\frac{s}{R}).
\label{eq3}
\end{align}
We investigate the integral mean of \( \varphi_{3}\). We have
\begin{align*}
\frac{{\rm d}}{{\rm d}t}\int^{\theta R}_{0}\varphi_{3}(s,t){\rm d}s
&=
(\varphi_{1s},\varphi_{2ss})-(\varphi_{2s},\varphi_{1ss})
-
(\varphi_{2ss},\sin(\frac{s}{R})) - (\varphi_{1ss},\cos(\frac{s}{R}))\\[3mm]
& \qquad +
\frac{1}{R}(\varphi_{2s},\cos(\frac{s}{R}))-\frac{1}{R}(\varphi_{1s},\sin(\frac{s}{R}))\\[3mm]
& =
(\varphi_{1s},\varphi_{2ss})-(\varphi_{2s},\varphi_{1ss})
+
\frac{2}{R}(\varphi_{2s},\cos(\frac{s}{R}))
-\frac{2}{R}(\varphi_{1s},\sin(\frac{s}{R}))\\[3mm]
&=
(\varphi_{1s},\varphi_{2ss})-(\varphi_{2s},\varphi_{1ss})
+
\frac{2}{R}\int^{\theta R}_{0}
\mbox{\mathversion{bold}$x$}^{R}_{s}(s,t)\cdot \mbox{\mathversion{bold}$\varphi$}_{s}(s,t)
{\rm d}s. \\[3mm]
&=
(\varphi_{1s},\varphi_{2ss})-(\varphi_{2s},\varphi_{1ss})
-\frac{1}{R}\|\mbox{\mathversion{bold}$\varphi$}_{s}(t)\|^{2},
\end{align*}
where we substituted (\ref{nostretch}) in the last equality.
Integrating with respect to \( t\) and applying the estimate in 
Lemma \ref{basic} yields
\begin{align}
\label{opt}
\big|
\int^{\theta R}_{0}\varphi_{3}(s,t){\rm d}s \big|
&\leq
\big|
\int^{\theta R}_{0}\varphi_{03}(s){\rm d}s \big|
+
C\|\mbox{\mathversion{bold}$\varphi$}_{0ss}\|^{2}t, \\[3mm]
& \leq C\big(
\|\varphi_{0,3}\| + \| \mbox{\mathversion{bold}$\varphi$}_{0ss}\|^{2}t\big)
\nonumber
\end{align}
where \( C>0\) depends only on \( \theta \) and \( R\). 
Furthermore, from the Poincar\'e--Wirtinger inequality, we have
\begin{align}
\|\varphi_{3}(t)-\overline{\varphi}_{3}(t)\|
\leq
C\|\varphi_{3s}(t)\|,
\label{poincare}
\end{align}
for all \( t>0\), where
\begin{align*}
\overline{\varphi}_{3}(t)=\frac{1}{\theta R}\int^{\theta R}_{0}
\varphi_{3}(s,t){\rm d}s.
\end{align*}
Combining (\ref{opt}), (\ref{poincare}), and Lemma \ref{basic}, we finally obtain
\begin{align*}
\| \varphi_{3}(t)\| \leq 
\frac{1}{\theta R}\big| \int^{\theta R}_{0}
\varphi_{3}(s,t){\rm d}s\big|
+
C\|\varphi_{3s}(t)\|
\leq
C\big( \| \varphi_{0,3}\|+ \| \mbox{\mathversion{bold}$\varphi$}_{0ss}\|
+\| \mbox{\mathversion{bold}$\varphi$}_{0ss}\|^{2}t\big)
\end{align*}
for all \(t>0\), and this finishes the proof of
Theorem \ref{th1}. \hfill \( \Box \).

%
%
%
%
%
\subsection{Proof of Theorem \ref{thopt}}
Proving the optimality of the estimate for \( \varphi_{3}\) is very 
straight forward. We construct a very specific initial perturbation 
which causes the perturbation to grow linearly with respect to time.

Let \( n\) be an arbitrary positive integer. 
For \( R>0\) and \( \theta \in (0,\pi) \), choose \( R_{n}>0 \) 
such that 
\begin{align*}
(2\pi n + \theta )R_{n}  = R \theta,
\end{align*}
in other words, \( \displaystyle R_{n}=\frac{R\theta}{2\pi n + \theta} \).
Next, we define a series of initial perturbations
\( \{ \mbox{\mathversion{bold}$\phi$}_{n}\}_{n=1}^{\infty} \)
by
\begin{align*}
\mbox{\mathversion{bold}$\phi$}_{n}(s)=
{}^{t}\big( R_{n}\cos (\frac{s}{R_{n}})-R\cos (\frac{s}{R}), \
R_{n}\sin (\frac{s}{R_{n}}) - R\sin(\frac{s}{R}), \
0 \big).
\end{align*}
Perturbing the initial arc-shaped filament by 
\(\mbox{\mathversion{bold}$\phi$}_{n}\)
corresponds to looping the arc into a circle with radius \( R_{n}\),
\( n\) times, before setting the end-points of the filament on the 
lower and upper planes.
We are exploiting the fact that vortex filaments described by the 
Localized Induction Equation can self-intersect, as well as 
intersect with other objects without consequences.  

We see from direct calculation that the unique solution 
\( \mbox{\mathversion{bold}$x$}_{n}\) of problem (\ref{slant3})
with initial datum \( \mbox{\mathversion{bold}$x$}^{R}_{0}+
\mbox{\mathversion{bold}$\phi$}_{n}\)
is explicitly given by
\begin{align*}
\mbox{\mathversion{bold}$x$}_{n}(s,t)=
{}^{t}\big( R_{n}\cos (\frac{s}{R_{n}}), \
R_{n}\sin (\frac{s}{R_{n}}), \
\frac{t}{R_{n}} \big).
\end{align*}
Hence, we have 
\begin{align*}
|x_{n,3}(s,t)-x^{R}_{3}(s,t)|
=\Big( \frac{1}{R_{n}}-\frac{1}{R}\Big)t
=
\frac{2\pi n t}{R\theta },
\end{align*}
which proves Theorem \ref{thopt}. \hfill \( \Box \)

\subsection{Remark on Theorem \ref{th2}}
Theorem \ref{th2} follows immediately from Theorem \ref{th1}.
Indeed, for \( R>0\) and \( \theta \in [\pi, 2\pi )\), let us consider the 
solution \( \mbox{\mathversion{bold}$x$}\) of the problem
(\ref{slant3}). If the initial perturbation 
is symmetric with respect to \( s=\frac{\theta R}{2}\) as defined in 
Definition \ref{sym}, we see that the perturbed initial filament has reflection symmetry 
with respect to \( s=\frac{\theta R}{2}\), and hence, the solution 
\( \mbox{\mathversion{bold}$x$}\) and the difference \( \mbox{\mathversion{bold}$\varphi$}=
\mbox{\mathversion{bold}$x$}-\mbox{\mathversion{bold}$x$}^{R} \) also has this 
symmetry. Following the arguments in the proof of 
Lemma \ref{raw}, we see that \( \mbox{\mathversion{bold}$x$}|
_{I_{\frac{\theta R}{2}}}\) is the solution of problem (\ref{slant3}) with 
\( \theta \) replaced with \( \frac{\theta }{2}\).
Since \( \frac{\theta }{2}\in (0, \pi )\), 
\( \mbox{\mathversion{bold}$\varphi$}|_{I_{\frac{\theta R}{2}}} \) 
satisfies all the estimates stated in Theorem \ref{th1}, from which the 
estimates for \( \mbox{\mathversion{bold}$\varphi$}\) follows.
%

%
%
%
%
\section{Long-time Behaviour of a Circular Vortex Filament}
\setcounter{equation}{0}
In this section, we utilize our estimates for the 
arc-shaped filament to prove estimates for a 
circular vortex filament.
To this end, we consider the following initial value problem
describing the motion of a perturbed circular vortex filament.
\begin{align}
\left\{
\begin{array}{ll}
\mbox{\mathversion{bold}$x$}_{t}=\mbox{\mathversion{bold}$x$}_{s}
\times \mbox{\mathversion{bold}$x$}_{ss}, & s\in \mathbf{T}, \ t>0, \\[3mm]
\mbox{\mathversion{bold}$x$}(s,0)=\mbox{\mathversion{bold}$x$}^{R}_{0}(s)
+ \mbox{\mathversion{bold}$\varphi$}_{0}(s), & s\in \mathbf{T},
\end{array}\right.
\label{ring}
\end{align}
where \( \mathbf{T}=\mathbf{R}/[0,L] \) and \( L>0\) is the circumference
of the initial circular vortex filament 
\( \mbox{\mathversion{bold}$x$}^{R}_{0}= 
{}^{t}\big( R\cos (\frac{s}{R}), \ R\sin (\frac{s}{R}), \ 0\big) \).
Note that we have used the same notation as the initial arc-shaped filament 
because the expression is the same and only the domain of the function is 
different.
Similar to the analysis of the arc-shaped filament, we set
\( \mbox{\mathversion{bold}$\varphi$}(s,t)=\mbox{\mathversion{bold}$x$}(s,t)
- \mbox{\mathversion{bold}$x$}^{R}(s,t)\) and investigate 
\( \mbox{\mathversion{bold}$\varphi$} \).
Here, \( \mbox{\mathversion{bold}$x$}^{R}(s,t)=
{}^{t}\big( R\cos (\frac{s}{R}), \ R\sin (\frac{s}{R}), \ \frac{t}{R}\big) \)
as before.
We see that \( \mbox{\mathversion{bold}$\varphi $}\) satisfies
\begin{align}
\left\{
\begin{array}{ll}
\mbox{\mathversion{bold}$\varphi $}_{t}=
\mbox{\mathversion{bold}$\varphi $}_{s}\times \mbox{\mathversion{bold}$\varphi $}_{ss}
+
\mbox{\mathversion{bold}$x$}^{R}_{s}\times \mbox{\mathversion{bold}$\varphi $}_{ss}
+
\mbox{\mathversion{bold}$\varphi $}_{s}\times \mbox{\mathversion{bold}$x$}^{R}_{ss},
& s\in \mathbf{T}, \ t>0, \\[3mm]
\mbox{\mathversion{bold}$\varphi $}(s,0)=\mbox{\mathversion{bold}$\varphi $}_{0}(s),
& s\in \mathbf{T}.
\end{array}\right.
\label{ringpert}
\end{align}
The time-global solvability of problem (\ref{ring}) is proved by 
Tani and Nishiyama in \cite{30}. Hence, we have the following existence
theorem.
\begin{Th}
For any \( \mbox{\mathversion{bold}$\varphi$}_{0}\in H^{4}(\mathbf{T})\)
satisfying \( |\mbox{\mathversion{bold}$x$}^{R}_{0s}(s)
+ \mbox{\mathversion{bold}$\varphi$}_{0s}(s)|=1 \) for all 
\( s\in \mathbf{T}\),
there exists a unique solution \( \mbox{\mathversion{bold}$x$}(s,t)\)
of problem {\rm (\ref{ring})} satisfying
\( |\mbox{\mathversion{bold}$x$}_{s}(s,t)|=1 \) for all
\( s\in \mathbf{T}\) and \( t>0\),
\begin{align*}
\mbox{\mathversion{bold}$x$}\in \bigcap ^{2}_{j=0}C^{j}\big([0,\infty);H^{4-2j}(\mathbf{T})\big), \quad \text{and}
\qquad 
\mbox{\mathversion{bold}$x$}_{s}\in L^{\infty}\big( 0,\infty; H^{3}(\mathbf{T})\big),
\end{align*}
which implies that \( \mbox{\mathversion{bold}$\varphi $}(s,t)=
\mbox{\mathversion{bold}$x$}(s,t)-\mbox{\mathversion{bold}$x$}^{R}(s,t) \)
satisfies {\rm (\ref{ringpert})},
\begin{align*}
\mbox{\mathversion{bold}$\varphi$}\in \bigcap ^{2}_{j=0}C^{j}\big([0,\infty);H^{4-2j}(\mathbf{T})\big), \quad \text{and}
\qquad 
\mbox{\mathversion{bold}$\varphi$}_{s}\in L^{\infty}\big( 0,\infty; H^{3}(\mathbf{T})\big).
\end{align*}
\label{ringex}

\end{Th}
Next, we establish what kind of initial perturbation we consider.
We define the following notation.
\begin{df}
Let \( k \geq 3 \) be an integer.
For an initial perturbation \( \mbox{\mathversion{bold}$\varphi$}_{0}
\in H^{4}(\mathbf{T}) \), we say that \( \mbox{\mathversion{bold}$\varphi$}_{0}\)
satisfies the \( k\)-reflective property if there exists
\( 0\leq s_{1}<s_{2}< \cdots < s_{k}<L \) such that the following holds.
\begin{description}
\item[\quad (i)] \( \mbox{\mathversion{bold}$\varphi$}_{0s}(s_{j}) 
=\mbox{\mathversion{bold}$0$} \) for all \( j\in \{ 1,2, \ldots , k\}\).
\item[\quad (ii)] For \( j\in \{ 1,2, \ldots , k\}\), denote
\( I_{j}=[s_{j},s_{j+1}]\), where \( s_{k+1}=s_{1}+L \), and we also denote
\( I_{k+1}=[s_{1}+L, s_{2}+L]\).
For all \( j\in \{ 2, 3, \ldots , k+1\}\), 
\( \mbox{\mathversion{bold}$\varphi$}_{0}\big|_{I_{j-1}\cup I_{j}}\) has 
reflection symmetry with respect to the plane perpendicular to 
\( {}^{t}(-\sin (\frac{s_{j}}{R}), \cos(\frac{s_{j}}{R}),0 ) \).
Here, \( \mbox{\mathversion{bold}$\varphi$}_{0}\big|_{I_{j-1}\cup I_{j}}\) is the 
restriction of \( \mbox{\mathversion{bold}$\varphi$}_{0} \) to 
two adjacent intervals.

\end{description}

\label{krefl}
\end{df}
Note that the definition of the \( k\)-reflective property necessitates that 
each interval \( I_{j}\) has the same length.
The above definition also implies that if 
\( \mbox{\mathversion{bold}$\varphi$}_{0}\)
satisfies the \( k\)-reflective property, the perturbed initial filament
\( \mbox{\mathversion{bold}$x$}_{0}= \mbox{\mathversion{bold}$x$}^{R}_{0}+\mbox{\mathversion{bold}$\varphi$}_{0}\)
satisfies (ii) of Definition \ref{krefl}. In other words,
\( \mbox{\mathversion{bold}$x$}_{0}\big|_{I_{j-1}\cup I_{j}} \) has reflection symmetry with respect to the 
plane perpendicular to 
\( {}^{t}(-\sin (\frac{s_{j}}{R}), \cos(\frac{s_{j}}{R}),0 ) \)
for all \( j\in \{ 2, 3, \ldots , k+1\}\). Furthermore, 
we see that 
\begin{align*}
 \mbox{\mathversion{bold}$x$}_{0s}(s_{j})=
{}^{t}\big( -\sin (\frac{s_{j}}{R}), \cos(\frac{s_{j}}{R}),0 \big)
\end{align*}
for all \( j\in \{ 1,2, \ldots , k\}\).

This means that, if \( \mbox{\mathversion{bold}$\varphi$}_{0}\) satisfies the 
\( k\)-reflective property, problem (\ref{ring}) can be divided into 
segments. Namely, consider the auxiliary problem for 
\( \mbox{\mathversion{bold}$y$}^{j}\) given by
\begin{align*}
\left\{
\begin{array}{ll}
\mbox{\mathversion{bold}$y$}^{j}_{t}=\mbox{\mathversion{bold}$y$}^{j}_{s}\times
\mbox{\mathversion{bold}$y$}^{j}_{ss}, & s\in I_{j}, \ t>0, \\[3mm]
\mbox{\mathversion{bold}$y$}^{j}(s,0)=\mbox{\mathversion{bold}$x$}_{0}(s), &
s\in I_{j}, \\[3mm]
\mbox{\mathversion{bold}$y$}^{j}_{s}(s_{j},t)
=\mbox{\mathversion{bold}$b$}^{j}, & t>0, \\[3mm]
\mbox{\mathversion{bold}$y$}_{s}^{j}(s_{j+1},t)
=\mbox{\mathversion{bold}$b$}^{j+1}, & t>0,
\end{array}\right.
\end{align*}
where \( \mbox{\mathversion{bold}$b$}^{j}
={}^{t}\big( -\sin (\frac{s_{j}}{R}), \cos(\frac{s_{j}}{R}),0 \big) \).
Note that the reflection symmetry of
\( \mbox{\mathversion{bold}$\varphi$}_{0} \) ensures that the necessary 
compatibility conditions for the above initial-boundary value problem 
are met.
Similar arguments as in the proof of Lemma \ref{raw} show that 
\begin{align*}
\mbox{\mathversion{bold}$x$}(s,t)=
\left\{
\begin{array}{ll}
\mbox{\mathversion{bold}$y$}^{1}(s,t), & \text{if} \ s\in I_{1}, \ t>0, \\[3mm]
\mbox{\mathversion{bold}$y$}^{2}(s,t), & \text{if} \ s\in I_{2}, \ t>0, \\[3mm]
\ \ \ \vdots & \ \\[3mm]
\mbox{\mathversion{bold}$y$}^{k}(s,t), & \text{if} \ s\in I_{k}, \ t>0, \\[3mm]
\end{array}\right.
\end{align*}
is the solution of problem (\ref{ring}). Hence, by applying 
the estimates obtained in Theorem \ref{th1} to each 
\( \mbox{\mathversion{bold}$y$}^{j}\), we obtain the following.
\begin{Th}
Let \( k\geq 3\) be an integer. For any \( \mbox{\mathversion{bold}$\varphi$}_{0}
\in H^{4}(\mathbf{T})\) satisfying
\( |\mbox{\mathversion{bold}$x$}^{R}_{0s}(s)+
\mbox{\mathversion{bold}$\varphi$}_{0s}(s)|=1\) for all \( s\in \mathbf{T}\)
and the \( k\)-reflective property, there exists 
\( C>0\) such that \( \mbox{\mathversion{bold}$\varphi$}
={}^{t}(\varphi_{1},\varphi_{2},\varphi_{3}) \)
given in Theorem {\rm \ref{ringex}} satisfies the following.
\begin{description}
\item[\quad (i)] For any \( t>0 \), 
\begin{align*}
\| {}^{t}(\varphi_{1}(t),
\varphi_{2}(t))\| + 
\|\mbox{\mathversion{bold}$\varphi $}_{s}(t)\|_{1}\leq C
\| \mbox{\mathversion{bold}$\varphi $}_{0ss}\|
\end{align*}
holds. 
\item[\quad (ii)] For any \( t>0 \),
\begin{align*}
\|\mbox{\mathversion{bold}$\varphi $}_{sss}(t)\|_{1} \leq 
C(\|\mbox{\mathversion{bold}$\varphi $}_{0ss}\|_{2} + \|\mbox{\mathversion{bold}$\varphi $}_{0ss}\|_{2}^{3})
\end{align*}
holds. 
\item[\quad (iii)] For any \( t>0 \),
\begin{align*}
\| \varphi_{3}(t) \| 
\leq C(\|\varphi_{0,3}\|+ \|\mbox{\mathversion{bold}$\varphi $}_{0ss}\|
+\|\mbox{\mathversion{bold}$\varphi $}_{0ss}\|^{2}t )
\end{align*}
holds.
\end{description}
\label{ringstab}
\end{Th}
We can further derive the following theorem from Theorem \ref{thopt}.
\begin{Th}
The estimate in {\rm (iii)} of Theorem {\rm \ref{ringstab}}
 is optimal in the following
sense. There exists \( C>0\) such that
for any integer \( n\geq 2\), the solution 
\( \mbox{\mathversion{bold}$x $}_{n}\) of problem {\rm (\ref{ring})} with
the initial perturbation 
given by
\begin{align*}
\mbox{\mathversion{bold}$\varphi $}_{n}(s)=
{}^{t}\big( R_{n}\cos (\frac{s}{R_{n}})-R\cos (\frac{s}{R}), \
R_{n}\sin (\frac{s}{R_{n}}) - R\sin(\frac{s}{R}), \
0 \big),
\end{align*}
where \( \displaystyle R_{n}= \frac{R}{n} \), satisfies
\begin{align*}
| x_{n,3}(s,t)-x^{R}_{3}(s,t) |
 = Cnt
\end{align*}
for all \( s\in \mathbf{T}\) and \( t\geq 0\).
\label{ringthopt}

\end{Th}
\begin{proof}
Since \( \mbox{\mathversion{bold}$\varphi $}_{n} \) satisfies the 
\( 3n\)-reflective property, we can divide the problem into 
\( 3n\) segments. We then apply Theorem \ref{thopt} to each segment to
obtain Theorem \ref{ringthopt}.
\end{proof}

%
%
%
%
\section{Concluding Remarks}
\setcounter{equation}{0}
We make some concluding remarks and compare our results to the 
behaviour of a vortex ring.
\subsection{On Theorem \ref{thopt}}
Theorem \ref{thopt} shows that the estimate for 
\( \varphi_{3}\) given in Theorem \ref{th1} is optimal. 
To show this, we have constructed an 
explicit initial perturbation which causes the perturbation to grow 
linearly with respect to time. The initial perturbation 
exploits the fact that the filament can overlap with itself, as well as intersect with 
other objects without any consequences. Although this kind of 
perturbation is within the assumptions of Theorem \ref{th1}, 
it is technical in nature and doesn't give much insight
on the motion of a thin vortex structure in the context of the 
Euler equation, which a vortex filament is meant to describe. 
This is partly due to the fact that stability estimates for 
\( \varphi_{1}\) and \( \varphi_{2}\) hold for a broad range of
perturbations. In fact, the assumptions of Theorem \ref{th1} are
mostly necessary conditions, such as compatibility conditions
and conditions that are needed to keep the problem setting 
consistent. Hence, a wide range of initial perturbations,
even ones that are not physically relevant, are included.

Given these situations, it would be interesting to see if the 
linear growth with respect to time is still optimal under 
restrictions on the initial perturbation 
which keeps the shape of the perturbed filament 
relevant from the perspective of the Euler equation.

One such possibility is to restrict the size of the perturbation to be 
small, which would result in the perturbed filament to remain close 
in shape to an arc. 
Analyzing the detailed shape of a vortex filament in general is 
a difficult problem, and is something the author would like to continue investigating
in the future.

%
%
%
%
\subsection{On the Stability of Circular Vortex Filaments}
Regarding the stability of circular vortex filaments, 
the same stability estimates and optimality as the arc-shaped filament 
is proved under the assumption that the 
initial perturbation satisfies the \( k\)-reflective property.
Even under this fairly strong restriction, the estimate 
for \( \varphi_{3}\) is optimal. Like in the case of the 
arc-shaped filament, the explicit perturbation 
used to prove the optimality has much to be desired, and the author would like to 
further investigate the stability of circular vortex filaments 
under perturbations which result in perturbed filaments with 
physically relevant shapes.

%
%
%
%
\subsection{Comparison with Known Results on the Euler Equation}
The stability estimates we obtained for the arc-shaped filament and 
circular filament suggest that similar behaviour may occur for the 
corresponding solutions of the Euler equation.
This is the case for certain types of vortex rings.
Choi \cite{79} and Choi and Jeong \cite{78} prove that a 
vortex ring is orbitally stable as the solution of the axisymmetric 
Euler equation. Furthermore, they prove that the perturbation 
in the translational direction, in which the ring travels, grows at most 
linearly with respect to \( t\). They also prove that this estimate is
optimal under their problem setting. This result is in direct correspondence 
with our result. Although the vortex ring considered in \cite{79} and \cite{78}
is considered a thick vortex ring, which is not a type of vortex structure
a vortex filament is meant to describe, it is interesting that 
such a clear correspondence can be observed.
%
%
%
%
%
\subsection{Possible Applications}
Jerrard and Seis \cite{85} prove that, given an initial smooth closed curve, 
the solution of the LIE starting from this initial curve is close,
in an appropriate sense, to the solution of the 
three-dimensional Euler equation where the initial flow is 
given by a vortex concentrated near the initial curve.
This is the first work which proved the validity of the LIE as a
model equation in the framework of the Euler equation.
More recently, D\'{a}vila, del Pino, Musso, and Wei \cite{86} proved 
a similar result for the helical vortex filament.

These results suggest that analyzing the motion of a vortex filament described 
by the LIE in detail can give information on the behaviour of the solution of the
three-dimensional Euler equation. This is another motivation for 
considering the motion of a vortex filament with physically relevant shapes.

\section*{Acknowledgements}
This work was supported in part by JSPS Grant-in-Aid for Early-Career
Scientists grant number 20K14348.

\vspace*{1cm}
\noindent
Masashi Aiki\\
Department of Mathematics\\
Faculty of Science and Technology, Tokyo University of Science\\
2641 Yamazaki, Noda, Chiba 278-8510, Japan\\
E-mail: a27120\verb|@|rs.tus.ac.jp

\end{document}